\documentclass{amsart}

\usepackage{amsmath,amssymb,amsthm,amscd}

\theoremstyle{plain}

\newtheorem{thm}{Theorem}[section]
\newtheorem{thmintro}{Theorem}
\newtheorem{prop}[thm]{Proposition}
\newtheorem{lemma}[thm]{Lemma}
\newtheorem{cor}[thm]{Corollary}

\theoremstyle{definition}

\newtheorem{defn}[thm]{Definition}
\newtheorem{ex}[thm]{Example}
\newtheorem{remark}[thm]{Remark}

\renewcommand{\o}[1]{\overline{#1}}

\newcommand{\C}{\mathbb C}

\newcommand{\F}{\mathbb F}
\newcommand{\G}{\mathcal G}

\renewcommand{\L}{\mathcal L}
\renewcommand{\O}{\mathcal O}
\renewcommand{\P}{\mathbb P}

\newcommand{\Q}{\mathbb Q}

\newcommand{\Z}{\mathbb Z}

\newcommand{\lra}{\longrightarrow}

\newcommand{\Fp}{\F_{p}}

\newcommand{\Zp}{\Z_{p}}
\newcommand{\Zpbar}{\o{\Z}_p}
\newcommand{\Fpbar}{\o{\F}_p}
\newcommand{\Qp}{\Q_{p}}
\newcommand{\Qpbar}{\o{\Q}_p}
\newcommand{\Qvbar}[1]{\o{\Q}_{#1}}
\newcommand{\Qbar}{\o{\Q}}
\newcommand{\Qn}{\Q_{n}}

\newcommand{\psmallmat}[4]{\left( \begin{smallmatrix} #1 & #2 \\ #3 & #4 \end{smallmatrix} \right)}
\newcommand{\pmat}[4]{\left( \begin{matrix} #1 & #2 \\ #3 & #4 \end{matrix} \right)}

\newcommand{\ve}{\varepsilon}
\newcommand{\rhobar}{\o{\rho}}
\newcommand{\chibar}{\o{\chi}}
\newcommand{\rhobarloc}[1]{\rhobar_{#1}\big|_{G_{\Qp}}}
\newcommand{\rhobarinert}[1]{\rhobar_{#1}\big|_{I_p}}
\newcommand{\mum}{\mu_{\min}}

\newcommand{\mump}{\mu_{\min}^+}

\newcommand{\mm}[2]{\mu_{\min}^{#2}(#1)}
\renewcommand{\t}[2]{\theta_{#1}(#2)}
\newcommand{\ti}[2]{\theta_{#1,i}(#2)}
\renewcommand{\tt}[2]{\vartheta_{#1}(#2)}
\newcommand{\ttn}[1]{\vartheta_n(#1)}
\newcommand{\tn}[1]{\theta_n(#1)}
\newcommand{\ttnf}{\vartheta_n(f)}
\newcommand{\ttv}[3]{\vartheta_{#1,#3}(#2)}
\newcommand{\tnf}{\theta_n(f)}
\newcommand{\tnfi}{\theta_{n,i}(f)}
\newcommand{\tni}[1]{\theta_{n,i}(#1)}

\newcommand{\lam}{\lambda}

\DeclareMathOperator{\Fil}{Fil}
\DeclareMathOperator{\GL}{GL}
\DeclareMathOperator{\SL}{SL}
\DeclareMathOperator{\Hom}{Hom}
\DeclareMathOperator{\ord}{ord}

\DeclareMathOperator{\Div}{Div}
\DeclareMathOperator{\Gal}{Gal}
\DeclareMathOperator{\Sym}{Sym}

\DeclareMathOperator{\Sel}{Sel}
\DeclareMathOperator{\cores}{cor}
\DeclareMathOperator{\irr}{irr}
\DeclareMathOperator{\rd}{red}

\newcommand{\inj}{\hookrightarrow}
\newcommand{\red}[1]{\o{#1}}
\newcommand{\Vbar}{\o{V}}
\newcommand{\vpb}{\o{\varphi}}
\newcommand{\Dp}{\big| \psmallmat{p}{0}{0}{1}}
\newcommand{\Dpv}[1]{~\big| \psmallmat{#1}{0}{0}{1}}
\newcommand{\GQ}{G_\Q}

\newcommand{\pin}[2]{\pi^{#1}_{#2}}
\newcommand{\nun}[2]{\cores^{#1}_{#2}}

\newcommand{\FV}[1]{\Fil^{#1}(V_g)}

\newcommand{\nop}{$p$-non-ordinary}
\renewcommand{\th}{\text{th}}

\title{Mazur--Tate elements of non-ordinary modular forms}
\author{Robert Pollack and Tom Weston}

\address[Robert Pollack]{Department of Mathematics, Boston University, Boston, MA}
\email[Robert Pollack]{rpollack@math.bu.edu}
\thanks{The first author was supported by NSF grant DMS-0701153 and a Sloan Research Fellowship.}

\address[Tom Weston]{Dept.\ of Mathematics, University of
Massachusetts, Amherst, MA}
\email[Tom Weston]{weston@math.umass.edu}
\thanks{The second author was supported by NSF grant DMS-0700359.}


\begin{document}

\begin{abstract}
We establish formulae for the Iwasawa invariants of Mazur--Tate elements of cuspidal eigenforms, generalizing known results in weight 2.  Our first theorem deals with forms of ``medium" weight, and our second deals with forms of small slope.  We give examples illustrating the strange behavior which can occur in the high weight, high slope case.
\end{abstract}

\maketitle

\section{Introduction}

Fix an odd prime $p$, and let $f$ denote a cuspidal eigenform of weight $k \geq 2$ and level $\Gamma := \Gamma_0(N)$ with $p \nmid N$.   Throughout this introduction, we assume for simplicity that $f$ has rational Fourier coefficients.  Let $\rhobar_f : G_\Q \to \GL_2(\Fp)$ denote the associated residual Galois representation which we assume to be irreducible.
If $f$ is a $p$-ordinary form, then the $p$-adic $L$-function $L_p(f)$ is an Iwasawa function, and one can associate to $f$ (analytic) Iwasawa invariants $\mu(f) = \mu(L_p(f))$ and $\lambda(f) = \lambda(L_p(f))$.  

If $f$ is \nop, then the situation is quite different as $L_p(f)$ is no longer an Iwasawa function, and one does not have associated $\mu$- and $\lambda$-invariants.   
However, when $f$ has weight 2, constructions of Kurihara and Perrin-Riou produce {\it pairs} of $\mu$- and $\lambda$-invariants denoted by $\mu^\pm(f)$ and $\lambda^\pm(f)$ (see also \cite{Pollackthesis} when $a_p(f) = 0$).
These invariants are defined by working with the Mazur--Tate elements
$$
\tnf \in \Zp[G_n]
$$ 
attached to $f$;
here $G_n = \Gal(\Qn/\Q)$
where $\Qn$ is the $n^\th$ layer of the cyclotomic $\Zp$-extension of $\Q$.
These elements interpolate the algebraic part of special values of the $L$-series of $f$; in fact,
$L_p(f)$ can be reconstructed as a limit of the $\tnf$.

To define the Iwasawa invariants in the non-ordinary weight 2 case, one shows that the sequences $\{ \mu(\t{2n}{f}) \}$ and $\{ \mu(\t{2n+1}{f}) \}$ stabilize as $n \to \infty$; the limit of these sequences are the invariants $\mu^+(f)$ and $\mu^-(f)$.
For the $\lambda$-invariants, the sequence $\{ \lambda(\tnf) \}$ is unbounded, but grows in a regular manner: the invariants $\lambda^\pm(f)$ have the property that (for sufficiently large $n$)
$$
\lambda(\tnf) = q_n + \begin{cases}
\lambda^+(f)  & \text{~if~} 2 \mid n \\
\lambda^-(f)  & \text{~if~} 2 \nmid n, 
\end{cases}
$$
where
$$
q_n = \begin{cases}
p^{n-1} - p^{n-2} + \cdots + p - 1  & \text{~if~} 2 \mid n \\
p^{n-1} - p^{n-2} + \cdots + p^2 - p  & \text{~if~} 2 \nmid n. 
\end{cases}
$$

In \cite{GV,EPW,GIP}, the behavior of $\mu$- and $\lambda$-invariants under congruences was studied in the ordinary case for arbitrary weights and in the non-ordinary case in weight 2.  For instance, in the ordinary case, it was shown that if the $\mu$-invariant vanishes for one form, then it vanishes for all congruent forms.  In particular, the vanishing of $\mu$ only depends upon the residual representation $\rhobar =\rhobar_f$; we write $\mu(\rhobar)=0$ if this vanishing occurs.  Completely analogous results hold in the weight 2 non-ordinary case.  

The $\lambda$-invariant can change under congruences, but this change is expressible in terms of explicit local factors.  In fact, when $\mu(\rhobar)=0$, there exists some global constant $\lambda(\rhobar)$ such that the $\lambda$-invariant of any form with residual representation $\rhobar$ is given by $\lambda(\rhobar)$ plus some non-negative local contributions at places dividing the level of the form.

An analogous theory exists on the algebraic side for ordinary forms and weight 2 non-ordinary forms.  These invariants are built out of Selmer groups, and enjoy the congruence properties described above.  
Furthermore, the Mazur--Tate elements should control the size and structure of the corresponding Selmer groups.  For instance, in the non-ordinary case, 
the main conjecture predicts that
\begin{equation}
\label{eqn:algan}
\dim_{\Fp} \Sel_p(f/\Qn)[p] = \lambda(\tnf)
\end{equation}
when $\mu^\pm(f) = 0$ (see \cite{GIP}).  Here $\Sel_p(f/\Qn)$ is the $p$-adic Selmer group attached to $f$ over the field $\Qn$.  Moreover, Kurihara \cite[Conjecture 0.3]{Kurihara} conjectures that the Fitting ideals of these Selmer groups are generated by Mazur--Tate elements.

Whether or not the equality in (\ref{eqn:algan}) extends to higher weight non-ordinary forms is unknown.  Little is known about the size and structure of Selmer groups in this case.  In this paper, we instead focus on the right hand side of (\ref{eqn:algan}), and via congruences in the spirit of \cite{GV,EPW}, we attempt to describe the Iwasawa invariants of Mazur--Tate elements for non-ordinary forms which admit a congruence to some weight 2 form.   

\subsection{Theorem for medium weight forms}
We offer the following theorem which describes the Iwasawa invariants for ``medium weight" modular forms (compare to Corollary \ref{cor:invmedweight} in the text of the paper).
We note that the form $g$ which appears below is $p$-ordinary if and only if $\rhobarloc{f}$ is reducible (see section \ref{sec:modpreps}).

\begin{thmintro}
\label{thm:A}
Let $f$ be an eigenform in $S_k(\Gamma)$ which is \nop, and such that
\begin{enumerate}
\item $\rhobar_f$ is irreducible of Serre weight 2,
\item 
$2 < k < p^2+1$,
\vspace{.1cm}
\item 
$\rhobarloc{f}$ is not decomposable.
\end{enumerate}
Then there exists an eigenform $g \in S_2(\Gamma)$ with $a_\ell(f)  \equiv a_\ell(g) \pmod{p}$ for all primes $\ell \neq p$ such that if $\rhobarloc{f}$ is reducible (resp.\ irreducible), then
\begin{enumerate}
\item 
\label{item:muA}
$\mu(\tnf) = 0$ for $n \gg 0 \iff \mu(g) = 0$ (resp.\ $\mu^\pm(g)=0$);
\vspace{.2cm}
\item
if the equivalent conditions of (\ref{item:muA}) hold, then
$$
\lam(\tnf) =   
p^n - p^{n-1} +
\begin{cases}
 \lam(g) & \text{~if~}\rhobarloc{f} \text{~is~reducible}, \\
q_{n-1} + \lam^{\text{-}\ve_n}(g) & \text{~if~}\rhobarloc{f} \text{~is~irreducible}, 
\end{cases}
$$ 
for $n \gg 0$;
here $\ve_n$ equals the sign of $(-1)^{n}$.
\end{enumerate}
\end{thmintro}

Assuming that $\rhobar_g$ is globally irreducible, it is conjectured that $\mu(g)=0$ when $g$ is ordinary, and  $\mu^\pm(g)=0$ when $g$ is non-ordinary (see \cite{Greenberg,PR}).  Thus, the equivalent conditions of part (\ref{item:muA}) in Theorem \ref{thm:A} conjecturally hold.  Further, by combining Theorem \ref{thm:A} with the results of \cite{EPW,GIP}, one can express $\lam(\tnf)$ in terms $\lambda(\rhobar)$ and local terms at primes dividing $N$.

We note that if either of the latter two hypotheses of Theorem \ref{thm:A} are removed, then there exist forms which do not satisfy the conclusions of this theorem.  In fact, there are examples of modular forms with weight as small as $p^2+1$ for which the $\mu$-invariant of $\tnf$ is positive for arbitrarily large $n$.  In these examples, there is an obvious non-trivial lower bound on $\mu$ which we now explain.

\subsection{Lower bound for $\mu$}
We can associate to $f$ its (plus) modular symbol 
$$
\varphi_f = \varphi_f^+ \in H^1_c(\Gamma,V_{k-2}(\Qp))^+ \cong 
\Hom_\Gamma\left(\Div^0(\P^1(\Q)),V_{k-2}(\Qp)\right)^+;
$$
here $V_{k-2}(\Qp)$ is the space of homogeneous polynomials of degree $k-2$ in two
variables $X$,$Y$ over $\Qp$.
We normalize this symbol so that it takes values in $V_{k-2}(\Zp)$, but not $pV_{k-2}(\Zp)$.  We then define
\begin{align*}
\mum(f) = \mump(f) 
&= \min_{D \in \Delta_0} \ord_p \left(  \varphi_f(D)\Big|_{(X,Y)=(0,1)}  \right)\\
&= \min_{D \in \Delta_0} \ord_p \left( \text{coefficient~of~}Y^{k-2}\text{~in~} \varphi_f(D)  \right).
\end{align*}

Let $\G_n = \Gal(\Q(\mu_{p^n})/\Q)$.
The $n^\th$ level Mazur--Tate element in $\Zp[\G_n]$ is given by
$$
\ttnf = \sum_{a \in (\Z/p^n\Z)^\times} c_a \cdot \sigma_a
$$
with
$$
c_a = \text{coefficient~of~}Y^{k-2}\text{~in~}\varphi_f\left( \{ \infty \} - \{ a/p^n \} \right)
$$
where $\sigma_a$ corresponds to $a$ under the standard isomorphism $\G_n \cong (\Z/p^n\Z)^\times$.
The element $\tnf$ is defined as the projection of $\tt{n+1}{f}$ under the natural map $\Zp[\G_{n+1}] \to \Zp[G_n]$.
It follows immediately that
$$
\mum(\tnf) \geq \mum(f).
$$

\subsection{Theorem for low slope forms}
The following theorem applies to modular forms of arbitrary weight, but with small slope (compare to Corollary \ref{cor:invlowslope}).  (Note that this is a non-standard usage of the term slope as we are considering the valuation of the eigenvalue of $T_p$ as opposed to $U_p$.)

\begin{thmintro}
\label{thm:B}
Let $f$ be an eigenform in $S_k(\Gamma)$ such that 
\begin{enumerate}
\item $\rhobar_f$ is irreducible of Serre weight 2,
\item 
\label{item:slopeB}
$0 < \ord_p(a_p) < p-1$,
\vspace{.1cm}
\item $\rhobarloc{f}$ is not decomposable.
\end{enumerate}
Then 
$$
\mum(f) \leq \ord_p(a_p).
$$
Further, there exists an eigenform $g \in S_2(\Gamma)$ with $a_\ell(f)  \equiv a_\ell(g) \pmod{p}$ for all primes $\ell \neq p$ such that if $\rhobarloc{f}$ is reducible (resp.\ irreducible), then
\begin{enumerate}
\item \label{item:muB}
$\mu(\tnf) = \mum(f)$ for $n \gg 0 \iff \mu(g) = 0$ (resp.\ $\mu^\pm(g)=0$);
\vspace{.2cm}
\item 
\label{item:lambdaB}
if the equivalent conditions of (\ref{item:muB}) hold and $n \gg 0$, then 
$$
\lam(\tnf) =   
p^n - p^{n-1} +
\begin{cases}
 \lam(g) & \text{~if~}\rhobarloc{f} \text{~is~reducible}, \\
q_{n-1} + \lam^{\text{-}\ve_n}(g) & \text{~if~}\rhobarloc{f} \text{~is~irreducible}.
\end{cases}
$$ 
\end{enumerate}
\end{thmintro}

Note that the conclusions of Theorem \ref{thm:B} are the same as the conclusions of Theorem \ref{thm:A} except that the $\mu$-invariants tend to $\mum(f)$ rather than to $0$.  

Condition (\ref{item:slopeB}) in Theorem \ref{thm:B} is necessary as there exist forms $f$ of slope $p-1$ which do not satisfy conclusion (\ref{item:lambdaB}) of the theorem.  In fact, the values $\lambda(\tnf) - q_{n+1}$ grow without bound for these forms.  We will discuss these exceptional forms after sketching the proofs of Theorems \ref{thm:A} and \ref{thm:B}.

\subsection{Sketch of proof of Theorem \ref{thm:A}}
First we consider the proof of Theorem \ref{thm:A} in the case when $k=p+1$.  The map 
\begin{align*}
V_{p-1}(\Zp) &\to \Fp \\
P(X,Y) &\mapsto P(0,1) \pmod{p}
\end{align*}
induces a map
$$
\alpha : H^1_c(\Gamma,V_{p-1}(\Zp)) \to H^1_c(\Gamma_0,\Fp) 
$$
where $\Gamma_0 = \Gamma_0(Np)$.  (Note that we have first composed with restriction to level $\Gamma_0$.)  The map $\alpha$ is equivariant for the full Hecke-algebra, where at $p$ we let the Hecke-algebra act on the source by $T_p$ and on the target by $U_p$.

By results of Ash and Stevens \cite[Theorem 3.4a]{AS}, we have that $\alpha(\varphi_f) \neq 0$.  The system of Hecke-eigenvalues of $\alpha(\varphi_f)$ then arises as the reduction of the system of Hecke-eigenvalues of some eigenform $h \in S_2(\Gamma_0)$ (see \cite[Proposition 2.5a]{AS}).   This form $h$ is necessarily non-ordinary, and since $h$ has weight 2, it is necessarily $p$-old.  Let $g$ be the associated form of level $\Gamma$, and let $\vpb_g$ denote the reduction mod $p$ of $\varphi_g$, the (plus) modular symbol attached to $g$.

A direct computation shows that $\vpb_g \Dp$ is a Hecke-eigensymbol for the full Hecke-algebra with the same system of Hecke-eigenvalues as $\alpha(\varphi_f)$.  (The analogous statement for $\vpb_g$ is false as this is an eigensymbol for $T_p$ and not for $U_p$.)  By mod $p$ multiplicity one, we then have
$$
\alpha(\varphi_f) = \vpb_g \Dp.
$$
Note that we can insist upon an equality here as $\varphi_g$ is only well-defined up to scaling by a unit.  This equality of modular symbols implies the following relation of Mazur--Tate elements:
$$
{\theta_n(f)} \equiv \nun{n}{n-1}({\theta_{n-1}(g)}) \pmod{p}
$$
where $\nun{n}{n-1} : \Fp[G_{n-1}] \to \Fp[G_n]$ is the corestriction map.  Since
$$
\mu(\nun{n}{n-1}(\theta)) = \mu(\theta)
~\text{~and~}~\lambda(\nun{n}{n-1}(\theta)) = p^n - p^{n-1} + \lambda(\theta),
$$
Theorem \ref{thm:A} follows when $k=p+1$.

To illustrate how the proof proceeds for the remaining weights in the range $p+1 < k < p^2+1$, we consider the case when $k=2p$.  If we can show that $\alpha(\varphi_f)  \neq 0$, then the above proof goes through verbatim.  So assume that $\varphi_f$ is in the kernel of $\alpha$.  Then results of Ash and Stevens \cite[Theorem 3.4c]{AS} imply that there is an eigenform $h \in S_{p-1}(\Gamma)$ such that $\rhobar_h \otimes \omega \cong \rhobar_f$;
here $\omega$ is the mod $p$ cyclotomic character.  
Since the weight of $h$ is less than $p+1$, Fontaine--Lafaille theory gives an explicit description of $\rhobarloc{h}$.  However, as long as $\rhobarloc{f}$ is indecomposable, this description contradicts the fact that $\rhobar_f \cong \rhobar_h \otimes \omega$ has Serre weight 2.  

We illustrate this argument once more when $k=3p-1$ to show the general picture.  In this case, if $\alpha(\varphi_f) = 0$, then there exists an eigenform $h \in S_{2p-2}(\Gamma)$ such that $\rhobar_h \otimes \omega = \rhobar_f$.  Since the weight of $h$ is out of the Fontaine--Lafaille range, we cannot immediately determine the structure of $\rhobarloc{h}$.  Instead, we consider the modular symbol $\varphi_h \in H^1_c(\Gamma,V_{2p-4}(\Zp))$.  If $\alpha(\vpb_h) \neq 0$, then as before, $h$ is congruent to a weight 2 form, and we can describe the structure of $\rhobarloc{h}$.  On the other hand, if $\alpha(\vpb_h) = 0$, then $h$ is congruent to some eigenform $h' \in S_{p-3}(\Gamma)$ such that $\rhobar_{h'} \otimes \omega \cong \rhobar_h$.  As the weight of $h'$ is small, we can determine the structure of $\rhobarloc{h'}$.  In either case, this description implies that $f$ cannot have Serre weight 2.

For the remaining weights, one proceeds similarly, inductively decreasing the weight of the form being considered.  These arguments work up until weight $p^2+1$ when in fact there can be an eigensymbol with Serre weight 2 in the kernel of $\alpha$.  

\subsection{Sketch of proof of Theorem \ref{thm:B}}
For the proof of Theorem \ref{thm:B}, we consider the following $\Gamma_0$-stable
filtration on $V_g(\Zp)$:
$$
\Fil^r(V_g) = 
\left\{ \sum_{j=0}^g b_j X^j Y^{g-j} \in V_g(\Zp) ~:~
p^{r-j} \mid b_j \text{~for~} 0 \leq j \leq r-1
\right\}.
$$
One computes that if $\varphi_f$ takes values in $j$-th step of this filtration, then $\ord_p(a_p)$ is at least $j$.  In particular,  if $r =  \ord_p(a_p) +1$, then $\varphi_f$ cannot take of all its values in $\Fil^r(V_{k-2})$.  The Jordan-Holder factors of $V_{k-2}(\Zp)/\Fil^r(V_{k-2})$ are all isomorphic to $\Fp$ with $\gamma = \psmallmat{a}{b}{c}{d} \in \Gamma_0$ acting by multiplication by $a^i \det(\gamma)^j$ for some $i$ and $j$.  

The image of $\varphi_f$ in $H^1_c(\Gamma_0,V_{k-2}(\Zp)/\Fil^r(V_{k-2}))$ is non-zero, and thus must contribute to the cohomology of one of these Jordan-Holder factors.  In particular, there exists a congruent eigenform $h$ of weight 2 on $\Gamma_0(N) \cap \Gamma_1(p)$.  Again by computing the possibilities for $\rhobarloc{h}$, one can determine the exact Jordan-Holder factor to which $\varphi_f$ contributes as long as $r$ satisfies the bound in the hypothesis (\ref{item:slopeB}) of Theorem \ref{thm:B}.  

As a result of this computation, one sees that the function
\begin{align*}
\Delta_0 &\to \Fp \\
D &\mapsto \frac{ \varphi_f(D)\Big|_{(X,Y)=(0,1)}}{p^{\mum(f)}} \pmod{p}
\end{align*}
is a modular symbol of level $\Gamma_0$.  One then proceeds as in the proof of Theorem \ref{thm:A} to construct a congruent weight 2 form, and then deduce the appropriate congruence of Mazur--Tate elements.

\subsection{An odd example}
We close this introduction with one strange example.  For $p=3$, there is an eigenform $f$ in $S_{18}(\Gamma_0(17),\Qvbar{3})$ which is an eigenform of slope 5 whose residual representation is isomorphic to the $3$-torsion in $X_0(17)$.  (Note that $X_0(17)[3]$ is  locally irreducible at $3$ as $X_0(17)$ is supersingular at 3.)  The form $f$ does not satisfy the hypotheses of Theorems \ref{thm:A} or \ref{thm:B} as its weight and slope are too big.  For this form, we can show that $\mum(f) = \mu(\tnf)=4$ for all $n \geq 0$, and 
$$
\lambda(\tnf) = p^n-p^{n-2} + q_{n-2}
$$
for $n \geq 2$.  This behavior of the $\lambda$-invariants is different from the patterns in Theorems \ref{thm:A} and \ref{thm:B} where the $\lambda$-invariants equal $q_{n+1} = p^n-p^{n-1} + q_{n-1}$ up to a bounded constant.  As explained in section \ref{sec:example}, this different behavior can be explained in terms of the failure of multiplicity one at level $Np^r$ with $r \geq 2$.  

\subsection{Outline}
The outline of the paper is as follows: we begin by
reviewing the definition of Mazur--Tate elements and their relation to
$p$-adic $L$-functions.  As our focus will be on the Mazur--Tate elements,
we then discuss finite-level Iwasawa invariants, recalling known results
in the ordinary case (section 3) and the weight $2$ non-ordinary case (section 4).   In section 5 (resp.\ section 6) we prove Theorem 1 (resp.\ Theorem 2) 
on Iwasawa invariants of Mazur--Tate elements for 
non-ordinary modular forms of medium weight (resp.\ low slope).
In section 7, we explain in detail an example of this odd behavior of $\lambda$-invariants for a form of high weight and slope.

~\\
{\it Acknowledgements}: We owe a debt to Matthew Emerton for numerous enlightening conversations on this topic.  We heartily thank Kevin Buzzard for several suggestions which led to the proof of Theorem \ref{thm:B}.  

~\\
\noindent
{\bf Notation:}
Throughout the paper we fix an odd prime $p$.
Let $\Zpbar$ denote the ring of integers of $\Qpbar$, and for $x \in \Zpbar$, let $\red{x}$ denote the image of $x$ in $\Fpbar$.  
For a finite extension $\O$ of $\Zp$, we write $\varpi$ for a uniformizer of $\O$, and $\F$ for its residue field.  We fix an embedding
$\Qbar \inj \Qpbar$.  For an integer $n$, we write $\ve_n$ for the sign
of $(-1)^n$.

\section{Mazur--Tate elements of modular forms}

Let $f$ be a cuspidal eigenform of weight $k$ on a congruence subgroup $\Gamma = \Gamma_0(N)$.  Our goal in this section
is to define the $p$-adic Mazur--Tate elements $\ttnf$
attached to $f$.  These are elements of the group ring
$\Zpbar[\G_n]$ for all $n \geq 1$; here
\begin{align*}
\G_{n} = \Gal(\Q(\mu_{p^n})/\Q) &\cong (\Z/p^n\Z)^\times \\
\sigma_a &\leftrightarrow a 
\end{align*}
where $\sigma_a(\zeta)=\zeta^a$ for $\zeta \in \mu_{p^n}$.
The utility of these elements is that they allow one to recover normalized
special values of twists of the $L$-function of $f$; see
Proposition~\ref{prop:interpolation} for a precise statement.

\subsection{Mazur--Tate elements}

Let $R$ be a commutative ring, and set $V_g(R) = \Sym^{g}(R^2)$ which we view as the space of homogeneous polynomials of degree $g$ with coefficients in $R$
in two variables
$X$ and $Y$.  We endow $V_g(R)$ with a right action of $\GL_2(R)$ by 
$$
(P | \gamma)(X,Y) = P( (X,Y) \gamma^*) = P(dX - cY,-bX+aY)
$$
for $P \in V_g(R)$ and $\gamma \in \GL_2(R)$.

Let $\Gamma \subseteq \SL_2(\Z)$ denote a congruence subgroup.  Recall the canonical isomorphism of Hecke-modules (see \cite[Proposition 4.2]{AS})
$$
H^1_c(\Gamma,V_g(R)) \cong \Hom_\Gamma\left(\Div^0(\P^1(\Q)),V_g(R)\right)
$$
where the target of the map equals the collection of additive maps
$$
\left\{ \varphi: \Div^0(\P^1(\Q)) \to V_g(R) ~:~ \varphi(\gamma D) | \gamma = \varphi(D) \text{~for all~} \gamma \in \Gamma \right\}.
$$
As this isomorphism is canonical, we will implicitly identify these two spaces
from now on; we refer to them as spaces of {\it modular symbols}.

For a modular symbol $\varphi \in H^{1}_c(\Gamma,V_g(R))$, we define
the associated {\it Mazur--Tate element} of level $n \geq 1$ by
\begin{equation} \label{eqn:stickel}
\tt{n}{\varphi} = \sum_{a \in (\Z/p^n\Z)^\times} \varphi\left( \{ \infty \} -
\{ a/p^n \} \right) \Big|_{(X,Y) = (0,1)} \cdot \sigma_a \in R[\G_n].
\end{equation}

When $R$ contains $\Zp$, we may decompose the Mazur--Tate elements
$\tt{n}{\varphi}$ as follows.  Write
$$\G_{n+1} \cong G_{n} \times (\Z/p\Z)^{\times}$$
with $G_{n}$ cyclic of order $p^n$.  Let $\omega : (\Z/p\Z)^{\times} \to
\Zp^{\times}$ denote the usual embedding of the $(p-1)^{\text{st}}$
roots of unity in $\Zp$.  For each $i$, $0 \leq i \leq p-2$, we obtain
an induced map $\omega^i : R[\G_{n+1}] \to R[G_n]$, and we define
$\tni{\varphi} = \omega^i(\tt{n+1}{\varphi})$.  We simply write
$\tn{\varphi}$ for $\theta_{n,0}(\varphi)$.  

\subsection{Modular forms}

One can associate to 
each eigenform $f$ in $S_k(\Gamma,\C)$ a modular symbol
$\xi_f$ in $H^1_c(\Gamma,V_{k-2}(\C))$ such that
$$
\xi_f(\{r\} - \{s\}) = 2 \pi i \int_s^r f(z) (zX+Y)^{k-2} dz
$$
for all $r,s \in \P^1(\Q)$; here we write $\{r\}$ for the divisor associated
to $r \in \Q$.
The symbol $\xi_f$ is a Hecke-eigensymbol with the same Hecke-eigenvalues as $f$.

The matrix $\iota := \psmallmat{-1}{0}{0}{1}$ acts as an involution on 
these spaces of modular symbols, and thus $\xi_f$ can be uniquely written as $\xi^+_f + \xi^-_f$ with $\xi^\pm_f$ in the $\pm1$-eigenspace of $\iota$.  By a theorem of Shimura \cite{Shimura}, there exists complex numbers $\Omega^\pm_f$ such that $\xi^\pm_f$ takes values in $V_{k-2}(K_f) \Omega^\pm_f$ where $K_f$ is the field of Fourier coefficients of $f$.    We can thus view
$\varphi^\pm_f := \xi^\pm_f / \Omega^\pm_f$ as taking values in 
$V_{k-2}(\Qpbar)$ via our fixed embedding $\Qbar \inj \Qpbar$.  
Set $\varphi_f = \varphi_f^+ + \varphi_f^-$, which of course depends upon the choices of $\Omega_f^+$ and $\Omega_f^-$.

Throughout this paper it will be crucial that we have normalized these
choices of periods appropriately.
For any $\varphi \in H^1_c(\Gamma,V_{k-2}(\Qpbar))$, define
$$
|| \varphi || := \max_{D \in \Delta_0} || \varphi(D) || 
$$
where for $P \in V_{k-2}(\Qpbar)$, $||P||$ is given by the maximum of the absolute values of the coefficients of $P$.   Let $\O_f$ denote the ring of integers of the completion of the image of
$K_f$ in $\Qpbar$.

\begin{defn}
We say that $\Omega^+_f$ and $\Omega^-_f$ are {\it cohomological periods} for $f$ (with respect to our fixed embedding $\Qbar \inj \Qpbar$), if $|| \varphi^+_f || = || \varphi^-_f || = 1$; that is, if each of $\varphi^+_f$ and
$\varphi^-_f$
takes values in $V_{k-2}(\O_f)$, and each takes on at least one value with at least one coefficient in $\O_f^\times$.  Such periods clearly always exist for each $f$  and are well-defined up to scaling by elements $\alpha \in K_f$ such that the image of $\alpha$ in $\Qpbar$ is a $p$-adic unit.
\end{defn}

We now write $\ttnf$ for the Mazur--Tate element $\tt{n}{\varphi_f}$
computed with respect to cohomological periods.  As before, we obtain Mazur--Tate elements
$$\tnfi \in \O_f[G_n]$$
for each $n \geq 1$ and $i$, $0 \leq i \leq p-2$.  We simply write
$\tnf$ for $\theta_{n,0}(f)$.

\begin{remark} 
We note that our choice of periods {\it forces} these Mazur--Tate elements to have integral coefficients.  This should be contrasted with the case of elliptic curves where the choice of the N\'eron period does not {\it a priori} guarantee integrality.
\end{remark} 
 
The following proposition describes the interpolation property of Mazur--Tate elements for primitive characters.

\begin{prop} \label{prop:interpolation}
If $\chi$ is a primitive Dirichlet character of conductor $p^n>1$, then 
$$
\chi(\ttnf) = \tau(\chi) \cdot \frac{L(f,\chibar,1)}{\Omega^\ve_f}
$$
where $\ve_f$ equals the sign of $\chi(-1)$.
\end{prop}

\begin{proof}
This proposition follows from \cite[(8.6)]{MTT}.
\end{proof}

\begin{remark}
We note that the classical Stickelberger element
$$
\vartheta_n = \frac{1}{p^n} \sum_{a \in (\Z/p^n\Z)^\times}
a \cdot \sigma_a^{-1} \in \Q[\G_n]
$$
has a similar interpolation property: for $\chi$ a primitive character on $\G_n$, 
$$
\chi(\vartheta_n) = -L(\chibar,0).
$$
\end{remark}

\subsection{Three-term relation}

Let $\pin{n}{n-1}: \O[G_n] \to \O[G_{n-1}]$ be the natural projection, and let $\nun{n}{n-1}: \O[G_{n-1}] \to \O[G_{n}]$ denote the corestriction map given by
$$\nun{n}{n-1}(\sigma) =
\sum_{\substack{\tau \mapsto \sigma \\ \tau \in G_n}} \tau$$
for $\sigma \in G_n$.

We then have the following three-term relation among various Mazur--Tate elements.  

\begin{prop}
\label{prop:Qseqn}
If $p \nmid N$, we have
\begin{equation}
\label{eqn:Qseqn}
\pin{n+1}{n}(\ti{n+1}{f}) = a_p \tnfi - p^{k-2} \nun{n}{n-1}(\ti{n-1}{f}).
\end{equation}
for $n \geq 1$ and any $i$.
\end{prop}

\begin{proof}
This proposition follows from \cite[(4.2)]{MTT} and a straightforward computation.
\end{proof}

\subsection{Some lemmas}

The following computations will be useful later in the
paper.

\begin{lemma}
\label{lemma:degen}
For $\varphi \in H^1_c(\Gamma,V_g(\O))$ and $n \geq 1$, we have
$$
\tni{\varphi \Dp} = p^g \cdot \nun{n}{n-1} \left( \ti{n-1}{\varphi}\right).
$$
\end{lemma}

\begin{proof}
We have
\begin{align*}
\tt{n}{\varphi \Dp} 
&= \sum_{a \in \G_n} \left(\varphi \Dp\right) \left( \{\infty\} - \{a/p^n\} \right)\Big|_{(X,Y)=(0,1)} \cdot \sigma_a \\
&= \sum_{a \in \G_n} \varphi  \left( \{\infty\} - \{a/p^{n-1}\} \right)\Big|_{(X,Y)=(0,p)} \cdot \sigma_a \\
&= p^g \cdot \sum_{a \in \G_n} \varphi  \left( \{\infty\} - \{a/p^{n-1}\} \right)\Big|_{(X,Y)=(0,1)} \cdot \sigma_a \\
&= p^g \cdot \nun{n}{n-1}(\tt{n-1}{\varphi}).
\end{align*}
Projecting to $\O[G_n]$ then gives the lemma.
\end{proof}

\begin{lemma}
\label{lemma:FE}
If $f$ is a newform on $\Gamma_0(N)$ of weight $k$, then
$$
\varphi_f \big| \psmallmat{0}{-1}{N}{0} = \pm N^{\frac{k}{2}-1} \varphi_f.
$$
\end{lemma}

\begin{proof}
First note that as $f$ is a newform, $w_N(f) = \pm f$, and thus
$$
N^{-k/2} z^{-k} f(-1/ Nz) = \pm f(z). 
$$
Computing, we have
\begin{align*}
\left(\xi_f \big| \psmallmat{0}{-1}{N}{0}\right)&(\{r\} - \{s\}) \\
&=\xi_f\left(\{-1/Nr\} - \{-1/Ns\}\right) \big| \psmallmat{0}{-1}{N}{0}
 \\
&= 2\pi i \int^{-1/Nr}_{-1/Ns} f(z)(-NzY+X)^{k-2} \\
&= \frac{2 \pi i}{N}  \int^{r}_{s} f(-1/Nz)(Y/z+X)^{k-2} z^{-2} dz &&(z \mapsto -1/Nz) \\
&= \pm N^{k/2-1} 2 \pi i \int^{r}_{s} f(z)(Y+zX)^{k-2} dz \\ 
&= \pm N^{k/2-1} \xi_f(\{r\} - \{s\}),
\end{align*}
and the lemma follows.
\end{proof}

\section{The $p$-ordinary case}

In this section, we first introduce Iwasawa invariants in finite-level group algebras, and then analyze the $\mu$- and $\lambda$-invariants of Mazur--Tate elements of $p$-ordinary forms.

\subsection{Iwasawa invariants in finite-level group algebras}

Fix a finite integrally closed extension $\O$ of $\Zp$ and
let $\Lambda := \varprojlim \O[G_n]$ denote the Iwasawa algebra.
Given $L \in \Lambda$, we may define Iwasawa invariants of $L$ as follows.
Fix an isomorphism $\Lambda \cong \O[[T]]$ and write
$L = \sum_{j=0}^\infty a_j T^j$; we then define
\begin{align*}
\mu(L) &= \min_j \ord_p(a_j) \\
\lambda(L) &= \min \{ j : \ord_p(a_j) = \mu(L) \}.
\end{align*}
(This definition is independent of the choice of isomorphism $\Lambda \cong \O[[T]]$.)
Here we normalize $\ord_p$ so that $\ord_p(p)=1$.
Note that under this normalization,  if $\O$ is a ramified extension of $\Zp$, then $\mu(L)$ need not be in $\Z$.

In fact,
Iwasawa invariants can also be defined in the finite-level group algebras $\O[G_n]$.  Indeed, for $\theta \in \O[G_n]$, if write $\theta = \sum_{\sigma \in G_n} c_\sigma \sigma$, we then define 
$$
\mu(\theta) = \min_{\sigma \in G_n} \ord_p(c_\sigma).
$$
To define $\lambda$-invariants, let $\varpi$ be a uniformizer of $\O$, and set $\theta' = \varpi^{-a} \theta$ with $a$ chosen so that $\mu(\theta') = 0$.  Let $\F$ be the residue field of $\O$, and let $\red{\theta'}$ denote the (non-zero) image of $\theta'$ under the natural map $\O[G_n] \to \F[G_n]$.  
All ideals of $\F[G_n]$ are of the form $I_n^j$ with $I_n$ the augmentation
ideal; we then define 
$$
\lambda(\theta) = \ord_{I_n}{\red{\theta'}} =  \max \{ j : \red{\theta'} \in I_n^j   \}.
$$

The following lemmas summarize some basic properties of these $\mu$ and $\lambda$-invariants.  For more details, see \cite[Section 4]{algkur}.

\begin{lemma}
\label{lemma:lambdainv}
Fix $L \in \Lambda$ and let $L_n$ denote the image of $L$ in $\O[G_n]$.  Then for $n \gg 0$, we have $\mu(L) = \mu(L_n)$ and $\lambda(L) = \lambda(L_n)$. 

\end{lemma}

\begin{lemma}
\label{lemma:nuninv}
For $\theta \in \O[G_{n-1}]$, we have
\begin{enumerate}
\item $\mu(\nun{n}{n-1}(\theta)) = \mu(\theta)$,
\item $\lambda(\nun{n}{n-1}(\theta)) = p^n - p^{n-1} + \lambda(\theta)$.
\end{enumerate}

\end{lemma}

\begin{lemma}
\label{lemma:pininv}
Fix $\theta \in \O[G_n]$.
\begin{enumerate}
\item If $\mu(\pin{n}{n-1}(\theta)) = 0$, then $\mu(\theta) = 0$.
\item If $\mu(\theta)=0$, then $\lambda(\pin{n}{n-1}(\theta)) = \lambda(\theta)$.
\end{enumerate}

\end{lemma}

\subsection{$p$-adic $L$-functions for $p$-ordinary forms}
\label{sec:pL}

The Mazur--Tate elements $\tnfi$ can be used to construct the $p$-adic $L$-function of $f$ in the $p$-ordinary case.  As this construction motivates much of what we do here,
we briefly digress to describe it.

We first fix some notation for the remainder of this paper.  Fix an integer
$N$ relatively prime to $p$ and set $\Gamma = \Gamma_0(N)$.  Also set
$\Gamma_0 = \Gamma_0(Np)$ and $\Gamma_1 = \Gamma_0(N) \cap \Gamma_1(p)$.

Let $f$ be a weight $k$ eigenform on $\Gamma$ which is ordinary at $p$.  Let $\alpha$ denote the unique unit root of $x^2-a_px +p^{k-1}$, and let $f_\alpha$ denote the $p$-ordinary stabilization of $f$ to $\Gamma_0$. 
The three-term relation of Proposition~\ref{prop:Qseqn} only has two terms when $p$ divides the level, and so the Mazur--Tate elements $\ti{n}{f_\alpha}$ attached to $f_\alpha$ satisfy 
$$
\pin{n}{n-1} (\ti{n}{f_\alpha}) = \alpha \cdot \ti{n-1}{f_\alpha}.
$$
If we set
$$
\psi_{n,i}(f_\alpha) = \frac{1}{\alpha^n} \ti{n}{f_\alpha},
$$
then $\{ \psi_{n,i}(f_\alpha) \}$ is a norm-coherent sequence, and thus an element of $\Lambda$.  This element is exactly $L_p(f,\omega^i)$, the $p$-adic $L$-function of $f$, twisted by $\omega^i$, and computed with respect to the periods $\Omega^\pm_{f_\alpha}$.

\subsection{Iwasawa invariants in the $p$-ordinary case}

Let $f$ continue to be a $p$-ordinary eigenform on $\Gamma$.  Set $\mu(f,\omega^i) = \mu(L_p(f,\omega^i))$ and $\lambda(f,\omega^i) = \lambda(L_p(f,\omega^i))$.  From the results of the last section and from Lemma \ref{lemma:lambdainv}, we have that
\begin{equation}
\label{eqn:pstabinv}
\mu(\ti{n}{f_\alpha})  = \mu(f,\omega^i) ~~~\text{~and~}~~~
\lambda(\ti{n}{f_\alpha})  = \lambda(f,\omega^i)
\end{equation}
for $n\gg0$.  Thus, the Iwasawa invariants of these ``$p$-stabilized" Mazur--Tate elements are extremely well-behaved in the ordinary case. 

One would hope to deduce similar information about the Iwasawa invariants of the $\tnfi$.  Unfortunately, they are not always as well-behaved as their $p$-stabilized counterparts as the following example illustrates.

\begin{ex}
Let $f = \sum_n a_n q^n$ denote the newform of weight 2 on $\Gamma_0(11)$ corresponding to the elliptic curve $E = X_0(11)$.   If $\vpb_f$ denotes the reduction of the modular symbol attached to $f$ modulo $5$, we have
\begin{equation}
\label{eqn:eis}
\vpb_f = \vpb_f \Dpv{5}.
\end{equation}
One verifies this relation by noting that $\vpb_f$ is (up to a non-zero scalar) the reduction of the Eisenstein boundary symbol defined by
$$
\varphi_{\text{eis}}(\{r/s\}) = \begin{cases} 0 \text{~if~}  \gcd(s,11) = 1 \\ 1 \text{~otherwise}
\end{cases}
$$ 
where $\gcd(r,s) = 1$.  Since $\varphi_{\text{eis}}$ satisfies (\ref{eqn:eis}) so does $\vpb_f$.

From repeated applications of Lemma \ref{lemma:degen} we now obtain
$$
\tnf \equiv \nun{n}{n-1}(\t{n-1}{f}) \equiv \dots \equiv \nun{n}{0}(\t{0}{f}) \pmod{5}.
$$
Moreover, a direct computation shows that $\t{0}{f}$ is a unit and thus $$
\mu(\tnf) = 0 ~\text{~and~}~\lambda(\tnf) = p^n-1
$$
for all $n \geq 0$.  Hence, the $\lambda$-invariants in this case are unbounded.  Here the deduction about $\lambda$-invariants comes from Lemma \ref{lemma:nuninv}.  We note that this is the maximal possible $\lambda$-invariant for any non-zero element of $\O[G_n]$.

There are several additional oddities in this example.
First, the Iwasawa invariants of the $\ti{n}{f_\alpha}$ must behave nicely, and in fact, assuming the main conjecture, we have
$$
\mu(\t{n}{f_\alpha}) = 1 ~\text{~and~}~\lambda(\t{n}{f_\alpha}) = 0
$$
for all $n \geq 0$.  In the process of $p$-stabilizing $f$, one considers the difference $\varphi_{f} - \frac{1}{\alpha} \varphi_f \Dpv{5}$.  However, since $a_5 = 1$, we have that $\alpha \equiv 1 \pmod{5}$, and thus by (\ref{eqn:eis}) this difference is divisible by $5$.  In particular, this means that the cohomological periods $\Omega^\pm_f$ and $\Omega^\pm_{f_\alpha}$ differ by a multiple of 5, and we can choose them so that
$$
\Omega^\pm_{f_\alpha} = 5 \Omega^\pm_f. 
$$  

From this, one might expect the $\mu$-invariants of the $\t{n}{f_\alpha}$ to be {\it lower} than the $\mu$-invariants of the $\tnf$.  However, numerically (for small $n$) one sees that 
$$
\tnf \equiv \frac{1}{\alpha} \nun{n}{n-1}(\t{n-1}{f}) \pmod{5^2} 
$$
so that
$$
\t{n}{f_\alpha} = \frac{1}{5} \left(  \tnf - \frac{1}{\alpha} \nun{n}{n-1}(\t{n-1}{f}) \right)
$$
is divisible by $5$.

Lastly, we mention that the N\'eron period of the elliptic curve $E$ in this case agrees with $\Omega^+_{f_\alpha}$ up to a 5-unit.
\end{ex}

The oddities of the above example arise as 
the residual representation
$$\rhobar_f : \GQ \to \GL_2(\Fp)$$
is globally reducible and $\mu(L_p(f))$ is positive. However, when we are not in this case, we verify now that the Iwasawa invariants of the $\tnfi$ are well-behaved.

We first check that cohomological periods are unchanged under $p$-stabilization so long as $\rhobar_f$ is irreducible.
Recall that for a multiple $M$ of $N$ and a divisor $r$ of $M/N$ there is
a natural degeneracy map
\begin{align*}
B_{M/N,r} : H^1_c(\Gamma_0(N),\Fp) &\to H^1_c(\Gamma_0(M),\Fp) \\
\varphi &\mapsto \varphi \big| \psmallmat{r}{0}{0}{1}.
\end{align*}
In particular, we define a map
$$
B_{p} : H^1_c(\Gamma,\Fp)^2 \to H^1_c(\Gamma_0,\Fp)
$$
by
$$
B_p(\psi_1,\psi_2) \mapsto B_{p,1}(\psi_1) + B_{p,p}(\psi_2).
$$

\begin{thm}[Ihara's lemma]
The kernel of 
$$
B_{p} : H^1_c(\Gamma,\Fp)^2 \to H^1_c(\Gamma_0,\Fp)
$$
is Eisenstein.
\end{thm}

\begin{proof}
See \cite{Ribet}.
\end{proof}

\begin{lemma}
\label{lemma:periods}
Let $f$ be a $p$-ordinary newform of weight $k$ and level $\Gamma$.  Let $\alpha$ denote the unit root of $x^2-a_px+p^{k-1}$, and let $f_\alpha$ denote the $p$-stabilization of $f$ to level $\Gamma_0$.   If $\rhobar_f$ is globally irreducible and $\Omega^\pm_f$ are a pair of cohomological periods for $f$, then $\Omega^\pm_f$ are also a pair of cohomological periods for $f_\alpha$. 
\end{lemma}

\begin{proof}
Since $f_\alpha(z) = f(z) - \beta f(pz)$,  where $\beta$ is the non-unit root of $x^2-a_px +p^{k-1}$, a direct computation shows that
$$
\xi^\pm_{f_\alpha} = \xi^\pm_f - \frac{1}{\alpha} \xi^\pm_f \Dp
$$
and thus
$$
\frac{\xi^+_{f_\alpha}}{\Omega^+_f} + \frac{\xi^-_{f_\alpha}}{\Omega^-_f} = \varphi_f - \frac{1}{\alpha} \varphi_f \Dp.
$$
To establish the lemma we need to show that the reduction of the above symbol is non-zero.

For $k>2$, suppose instead that $\vpb_f=\frac{1}{\alpha}\cdot\vpb_f \Dp$.
The only non-zero coefficients in the values of $\vpb_f \Dp$ occur in the $X^{k-2}$ coefficients, and thus the same is true for $\vpb_f$.  But, by Lemma \ref{lemma:FE}, the vanishing of the coefficients of $Y^{k-2}$ implies the vanishing of the coefficients of $X^{k-2}$.  Thus, $\vpb_f =0$ which is a contradiction.

For $k=2$, consider the obviously injective map
\begin{align*}
j:  H^1_c(\Gamma,\Fp) &\to H^1_c(\Gamma,\Fp)^2 \\
 \varphi &\mapsto \left(\varphi,-\frac{1}{\alpha} \cdot \varphi\right).
\end{align*}
As $\rhobar_f$ is irreducible, by Ihara's lemma, $(B_p \circ j)(\vpb_f) \neq 0$, and thus, 
$$
\vpb_f-\frac{1}{\alpha}\cdot\vpb_f \Dp \neq 0
$$
as desired.
\end{proof}

We are now in a position to understand the Iwasawa invariants of
$\tnfi$ for $f$ $p$-ordinary with $\rhobar_f$ irreducible.

\begin{prop}
\label{prop:invord}
Assume that $\mu(L_p(f,\omega^i)) = 0$ and that $\rhobar_f$ is irreducible.  Then for $n \gg 0$, we have
$$
\mu(\tnfi) = 0 ~\text{~and~}~ \lambda(\tnfi) = \lambda(L_p(f,\omega^i)).
$$
\end{prop}

\begin{proof}
By Lemma \ref{lemma:periods},
$$
\varphi_{f_\alpha} = \varphi_f - \frac{1}{\alpha} \varphi_f \Dp,
$$
and hence,
\begin{equation}
\label{eqn:pstab}
\ti{n}{f_\alpha} = \tnfi - \frac{1}{\alpha} \nun{n}{n-1} (\ti{n-1}{f}).
\end{equation}
Since we are assuming that $\mu(\ti{n}{f_\alpha}) = 0$ for $n \gg 0$, by (\ref{eqn:pstab}) there exist sufficiently large $n$ for which $\mu(\tnfi)=0$.

If $k>2$, the argument proceeds as follows: the three-term relation of Proposition~\ref{prop:Qseqn} implies that if $\mu(\tnfi)=0$ for one $n$, then $\mu(\ti{m}{f})=0$ for all $m>n$ as desired.  For the $\lambda$-invariants,
by Lemma \ref{lemma:nuninv}, 
$$
\lambda(\nun{n}{n-1} (\ti{n-1}{f})) \geq p^n - p^{n-1},
$$
and thus for $n$ large enough, 
$$
\lambda(\ti{n}{f_\alpha}) < \lambda(\nun{n}{n-1} (\ti{n-1}{f})).
$$
By (\ref{eqn:pstab}) and (\ref{eqn:pstabinv}), we then have
$$
\lambda(\tnfi) = \lambda(\ti{n}{f_\alpha}) = \lambda(L_p(f,\omega^i)) 
$$
as desired.

For the case $k=2$, one must argue more carefully because the three-term relation does not guarantee the vanishing of $\mu(\tnfi)$ for all $n$ if one knows the vanishing for a single $n$.  From (\ref{eqn:pstab}) we do know that there exists sufficiently large $n$ such that $\mu(\tnfi)=0$.  For such a sufficiently large $n$, (\ref{eqn:pstab}) implies that 
$$
\lambda(\tnfi) = \lambda(\ti{n}{f_\alpha}) 
$$ 
as before.  Hence, 
$$
\lambda(\tnfi) \neq \lambda(\nun{n}{n-1}(\t{n-1}{f}))
$$ 
which implies that the reduction of these two elements are not equal.  In particular, by (\ref{eqn:Qseqn}) of Proposition~\ref{prop:Qseqn}, $\mu(\pin{n+1}{n} (\t{n+1}{f}))=0$, and thus by Lemma \ref{lemma:pininv}, $\mu(\t{n+1}{f})=0$ as desired.  Thus inductively, $\mu(\tnfi)$ vanishes for all sufficiently large $n$.  Finally, the statement about $\lambda$-invariants follows just as in the $k>2$ case.
\end{proof}

\section{The non-ordinary case}

In the non-ordinary case, the polynomial $x^2-a_px +p^{k-1}$ has no unit root.  Thus, the construction of $p$-adic $L$-functions described in section \ref{sec:pL} does not yield integral power series.  Indeed, if $\alpha$ is either root of this quadratic, then dividing by powers of $\alpha$ introduces $p$-adically unbounded denominators.  In the non-ordinary case, we therefore focus our attention on the elements $\tnfi$, rather than on passing to a limit to construct an unbounded $p$-adic $L$-function.  

\subsection{Known results in weight 2}

For modular forms of weight 2, the Iwasawa invariants of the associated Mazur--Tate elements were studied in detail by Kurihara \cite{Kurihara} and Perrin-Riou \cite{PR}. We summarize their results in the following theorem.

\begin{thm}
\label{thm:wt2}
Let $i$ be an integer with $0 \leq i \leq p-2$.
\begin{enumerate}
\item There exist constants $\mu^\pm(f,\omega^i) \in \Z^{\geq 0}$ such that for $n \gg 0$,
$$
\mu(\ti{2n}{f}) = \mu^+(f,\omega^i) \text{~~and~~} 
\mu(\ti{2n+1}{f}) = \mu^-(f,\omega^i).
$$
\item If $\mu^+(f,\omega^i) = \mu^-(f,\omega^i)$, then there exist constants $\lambda^\pm(f,\omega^i) \in \Z^{\geq 0}$ such that for $n \gg 0$,
$$
\lambda(\tnfi) = q_n + \begin{cases} 
 \lambda^+(f,\omega^i) & i \text{~~even} \\
 \lambda^-(f,\omega^i) & i \text{~~odd}
\end{cases}
$$
where
$$
q_n  = \begin{cases} 
p^{n-1} - p^{n-2} + \cdots + p - 1  & \text{~if~} n \text{~~even} \\
p^{n-1} - p^{n-2} + \cdots + p^2 - p  & \text{~if~} n \text{~~odd}.
\end{cases}
$$
\end{enumerate}
\end{thm}

\begin{remark}
\noindent
\begin{enumerate}
\item Perrin-Riou conjectured \cite[6.1.1]{PR} that $\mu^+(f,\omega^i)=\mu^-(f,\omega^i)=0$ (see also \cite[Conjecture 6.3]{Pollackthesis}).  This is an analogue of Greenberg's conjecture on the vanishing of $\mu$ in the ordinary case.  Indeed, Greenberg conjectures that $\mu$ vanishes whenever $\rhobar_f$ is irreducible; if $f$ has weight 2 and is \nop,  then $\rhobar_f$ is always irreducible.  

\item In \cite{GIP}, the assumption that $\mu^+(f,\omega^i)=\mu^-(f,\omega^i)$ is removed, but the resulting formula for $\lambda$ is slightly different in some cases when $\mu^+(f,\omega^i) \neq \mu^-(f,\omega^i)$.  However, since this case is conjectured to never occur, and in this paper we will only use these formulas when $\mu^\pm(f,\omega^i)=0$, we will not go further into this complication.

\item Unlike the ordinary case, the $\lambda$-invariants of these non-ordinary Mazur--Tate elements always grow without bound because of the presence of the $q_n$ term which is ${\text O}(p^{n-1})$.  
\end{enumerate}

\end{remark}

\begin{proof}[Proof of Theorem \ref{thm:wt2}]
In \cite{PR}, it is proven that any sequence of elements $\theta_n \in \O[G_n]$ which satisfy the three-term relation of
Proposition~\ref{prop:Qseqn} satisfy the conclusions of this theorem.  To give the spirit of these arguments, we give a proof here in the case when $\mu(\tnfi) = 0$ for $n \gg 0$.  

Since $a_p$ is not a unit, the three-term relation implies that
\begin{equation}
\label{eqn:rel}
\pin{n+1}{n}(\ti{n+1}{f}) \equiv \nun{n}{n-1}(\ti{n-1}{f}) \pmod{\varpi}.
\end{equation}
Thus for $n$ large enough we have
\begin{align*}
\lambda(\ti{n+1}{f}) &= \lambda(\pin{n+1}{n}(\ti{n+1}{f})) && (\text{by~ Lemma~} \ref{lemma:pininv})\\
&= \lambda(\nun{n}{n-1}(\ti{n-1}{f})) && (\text{by~} (\ref{eqn:rel}))\\
&= p^n - p^{n-1} + \lambda(\ti{n-1}{f}) && (\text{by~Lemma~} \ref{lemma:nuninv}).
\end{align*}
Proceeding inductively then yields the theorem.
\end{proof}

\subsection{Differences in weights greater than 2}

To compare with the case of weight $2$, we note that when $f$ is of any weight $k > 2$, then the three-term relation takes the form
$$
\pin{n+1}{n}(\ti{n+1}{f}) = a_p \tnfi - p^{k-2} \nun{n}{n-1}(\ti{n-1}{f}).
$$
The factor of $p^{k-2}$ in the third term prevents the arguments of the previous section from going through.  Indeed, the right hand side of the above equation vanishes mod $\varpi$, and one cannot make general deductions about the Iwasawa invariants of such sequences unlike the case when $k=2$.  Instead, the strategy of this paper is to make a systematic study of congruences between Mazur--Tate elements in weight $k$ and in weight $2$, and then to make deductions about Iwasawa invariants by invoking Theorem \ref{thm:wt2}.

\subsection{Lower bound for $\mu$}

We note that there is an obvious lower bound for $\mu$-invariants of Mazur--Tate elements in weights greater than 2.  For $\varphi \in H^1_c(\Gamma,V_{k-2}(\O))$, set
$$
\mum(\varphi)= \min_{D \in \Delta_0} \ord_p \left(  \varphi(D)\Big|_{(X,Y)=(0,1)}   \right);
$$
thus $\mum(\varphi)$ is the minimum valuation of the coefficients of $Y^{k-2}$ in the values of $\varphi$.   We write $\mm{f}{\pm}$ for $\mum(\varphi_f^\pm)$. 

\begin{prop}
We have that
\begin{enumerate}
\item $\mm{f}{\pm} < \infty$,
\item $\mu(\tnfi) \geq \mm{f}{\ve_i}$.
\end{enumerate}
\end{prop}

\begin{proof}
For the first part, if $\mm{f}{\pm} = \infty$, then $\tnf$ vanishes for every $n$.  By Proposition \ref{prop:interpolation}, we then have that $L(f,\chi,1) = 0$ for every character $\chi$ of conductor a power of $p$.  But this contradicts a non-vanishing theorem of Rohrlich \cite{Rohrlich}.

The second part is immediate as $\tnfi$ is constructed out of the coefficients of $Y^{k-2}$ of certain values of $\varphi_f^{\ve_i}$.
\end{proof}

Recall that $\varphi_f$ is normalized so that all of its values have coefficients which are integral and at least one which is a unit.  Thus, when $k=2$, by definition $\mm{f}{\pm}$ is always 0.  However, when $k>2$, it is possible that the coefficient of $Y^{k-2}$ in every value of $\varphi_f$ is a non-unit, and that the required unit coefficient occurs in another monomial;  in this case $\mm{f}{\pm}$ would be positive.

\subsection{A map from weight $k$ to weight $2$}
\label{sec:alpha}

In this section we discuss a map from weight $k$ modular symbols to weight $2$ modular symbols over $\Fp$ introduced by Ash and Stevens in \cite{AS}.  
Set 
$$
S_0(p) := \left\{ \psmallmat{a}{b}{c}{d} \in M_2(\Z) : ad - bc \neq 0, ~ p \mid c, ~ p \nmid a \right\},
$$
$g=k-2$, and $\Vbar_g = V_g(\F)$.

\begin{lemma}
\label{lemma:alpha}
For $g>0$ and $g \equiv 0 \pmod{p-1}$, the map 
\begin{align*}
\Vbar_g &\lra \F \\
P(X,Y) &\mapsto P(0,1)
\end{align*}
is $S_0(p)$-equivariant, and thus induces a Hecke-equivariant map
$$
\alpha: H^1_c(\Gamma,\Vbar_g) \lra H^1_c(\Gamma_0,\F).
$$
\end{lemma}

\begin{remark}
By Hecke-equivariant we mean the standard concept away from $p$, and at $p$, we mean that $\alpha$ intertwines the action of $T_p$ on the source with $U_p$ on the target. 
\end{remark}

\begin{proof}
This lemma follows from a straightforward computation.  We note that the Hecke-equivariance at $p$ follows from the fact that for $P \in \Vbar_g$ and $g>0$,
$$
\left( P \big| \psmallmat{p}{0}{0}{1} \right)\Big|_{(X,Y)=(0,1)} = 0.
$$\end{proof}

The following simple lemma is the key to our approach of comparing Mazur--Tate elements of weight $k$ and weight 2.

\begin{lemma}
\label{lemma:alphastick}
For $\varphi \in H^1_c(\Gamma,V_{k-2}(\O))$,
$$
\ttn{\alpha(\vpb)} = \ttn{\vpb} = \red{\ttn{\varphi}} \text{~in~} \F[\G_n],
$$
where $\vpb$ is the reduction of $\varphi$ modulo $\varpi$.
\end{lemma}

\begin{proof}
The first equality is true as these Mazur--Tate elements depend only on the coefficients of $Y^{k-2}$ in the values of $\vpb$, and the map $\alpha$ preserves
these coefficients.  The second equality is clear. 
\end{proof}

The following lemma gives the analogue for modular symbols of the $\theta$-operator for mod $p$ modular forms.  In what follows, if $M$ is a $S_0(p)$-module, then $M(1)$ is the determinant twist of $M$; for a Hecke-module $M$, the Hecke-operator $T_n$ acts on $M(1)$ by $nT_n$.

\begin{lemma}
The map 
\begin{align*}
\Vbar_{g-p-1}(1) &\lra \Vbar_g \\
P(X,Y) &\mapsto (X^pY-XY^p) \cdot P(X,Y)
\end{align*}
is $S_0(p)$-equivariant, and thus induces a Hecke-equivariant map
$$
\theta : H^1_c(\Gamma,\Vbar_{g-p-1})(1) \lra H^1_c(\Gamma_0,\Vbar_g).
$$
\end{lemma}

\begin{proof}
This is a straightforward computation.
\end{proof}

Lastly, we note that the kernel of $\alpha$ is given by precisely the symbols with positive $\mum$.

\begin{lemma}
\label{lemma:mum}
We have $\mum(\varphi) > 0 \iff \alpha(\vpb) = 0$.
\end{lemma}

\begin{proof}
We have $\alpha(\vpb)=0$ if and only if all of the coefficients of $Y^{k-2}$ occurring in values of $\varphi$ are divisible by $\varpi$, which is equivalent to $\mum(\varphi) > 0$.
\end{proof}

\subsection{Review of mod $p$ representations of $G_{\Qp}$}
\label{sec:modpreps}

For use in the following sections, we recall the possibilities for
the local residual representation of a modular form of small weight.

Let 
$\rhobar_p : G_{\Qp} \to \GL_2(\Fpbar)$ be an arbitrary continuous residual
representation of the absolute Galois group of $\Qp$.
If $\rhobar_p$ is irreducible, then $\rhobar_p \big|_{I_p}$ is tamely ramified; here $I_p$ denotes the inertia subgroup of $G_{\Qp}$.   Moreover, we have 
$$
\rhobar_p \big|_{I_p} \cong \omega_2^t \oplus \omega_2^{pt}
$$
where $\omega_2$ is a fundamental character of level 2 and $1 \leq t \leq p^2-1$ with $p+1 \nmid t$.  The integer $t$ uniquely determines $\rhobar_p \big|_{I_p}$ and we write $I(t)$ for this representation.  We note that $I(t) \cong I(pt)$.

If $\rhobar_p$ is reducible, then 
$$
\rhobar_p \big|_{I_p} \cong \pmat{\omega^a}{*}{0}{\omega^b}
$$
where $\omega$ is the mod $p$ cyclotomic character.

\begin{thm}
\label{thm:reps}
Let $f$ be an eigenform on $\Gamma$ of weight $k$ with $\rhobar_f$ irreducible.
\begin{enumerate}
\item
If $f$ is $p$-ordinary, then $\rhobarloc{f}$ is reducible and $\rhobarinert{f} \cong \pmat{\omega^{k-1}}{*}{0}{1}$.
\item
If $f$ is \nop{} and $2 \leq k \leq p+1$, then $\rhobarloc{f}$ is irreducible and $\rhobarinert{f} \cong I(k-1)$.
\end{enumerate}
\end{thm}

\begin{proof}
See \cite[Remark 1.3]{BG} for a thorough discussion of references for these results.
\end{proof}

The following lemma will be useful later in the paper.

\begin{lemma}
\label{lemma:nebenreps}
If $f$ is an eigenform in $S_2(\Gamma_1,\omega^j,\Qpbar)$ with $\rhobar_f$ irreducible and $0 \leq j \leq p-2$, then 
$$
\rhobarinert{f} \cong
\begin{cases}
I(j+1) &\text{~if~}\rhobarloc{f}\text{~is~irreducible,} 
\vspace{.1cm}\\
\pmat{\omega^{j+1}}{*}{0}{1}\text{~or~}
\pmat{\omega}{*}{0}{\omega^j}
&\text{~if~}\rhobarloc{f}\text{~is~reducible.} 
\end{cases}
$$
\end{lemma}

\begin{proof}
Consider the modular symbol $\vpb_f \in H^1_c(\Gamma_1,\F)^{(\omega^j)}$.  By \cite[Theorem 3.4(a)]{AS}\footnote{We note that in \cite{AS} the cohomology groups considered are not taken with compact support.  However, the difference between $H^1$ and $H^1_c$ is Eisenstein, and since we are assuming our forms have globally irreducible Galois representations, this difference does not affect our arguments.}, the system of eigenvalues of $\vpb_f$ occurs either in $
H^1_c(\Gamma,\Vbar_j)$ or $H^1_c(\Gamma,\Vbar_{p-1-j})(j)$.
By \cite[Proposition 2.5]{AS}, there then exists either an eigenform $g \in S_{j+2}(\Gamma,\Qpbar)$ with $\rhobar_f \cong \rhobar_g$ or an eigenform $g \in S_{p+1-j}(\Gamma,\Qpbar)$ with $\rhobar_f \cong \rhobar_g \otimes \omega^j$.

In the first case, by Theorem \ref{thm:reps}, $\rhobarinert{f}$ is equal to either $I(j+1)$ or $\psmallmat{\omega^{j+1}}{*}{0}{1}$, and, in the second case, $\rhobarinert{g}$ is equal to either $I(p-j)$ or $\psmallmat{\omega^{p-j}}{*}{0}{1}$.  In the latter case,
$$
\rhobarinert{f} \cong I(p-j) \otimes \omega^j \cong I(p-j+j(p+1)) \cong
I(pj+p) \cong I(j+1)
$$
or
$$
\rhobarinert{f} \cong \psmallmat{\omega^{p-j}}{*}{0}{1} \otimes \omega^j = \psmallmat{\omega}{*}{0}{\omega^j}.
$$
\end{proof}

\section{The non-ordinary case for medium weights}
\label{sec:medweight}

In this section, we will prove a theorem about the Iwasawa invariants of Mazur--Tate elements in weights $k$ such that $2 < k < p^2+1$.  For $f \in S_k(\Gamma,\Qpbar)$ a normalized eigenform, recall that $\O := \O_f$ denotes the ring of integers of the finite extension of $\Qp$ generated by the Fourier coefficients of $f$, and $\F := \F_f$ denotes the residue field of $\O$.

\subsection{Statement of theorem}

\begin{thm}
\label{thm:medweight}
Let $f$ be an eigenform in $S_k(\Gamma,\Qpbar)$ which is \nop, and such that
\begin{enumerate}
\item \label{hyp:irred}
$\rhobar_f$ is irreducible,
\vspace{.1cm}
\item \label{hyp:weight}
$2 < k < p^2+1$,
\vspace{.1cm}
\item \label{hyp:repn}
$k(\rhobar_f) = 2$ and $\rhobarloc{f}$ is not decomposable.
\end{enumerate}
Then
\begin{enumerate}
\item $\mm{f}{+}=\mm{f}{-}=0$, and
\item there exists an eigenform $g \in S_2(\Gamma)$ with $\red{a_\ell(f)} = \red{a_\ell(g)}$ for all primes $\ell \neq p$, and a choice of cohomological periods $\Omega_f, \Omega_g$ such that 
$$
\red{\ttnf} = \nun{n}{n-1}\bigl(\red{ \tt{n-1}{g}}\bigr) \text{~in~} \F[\G_n].
$$
\end{enumerate}
\end{thm}

\begin{remark}
\label{rmk:medweight}
\noindent
\begin{enumerate}
\item
\label{part:gord}
In the notation of the above theorem, we have that $g$ is ordinary at $p$ if and only if $\rhobarloc{f}$ is reducible.  Indeed, $\rhobar_f \cong \rhobar_g$, and since $g$ has weight 2, Theorem \ref{thm:reps} implies that  $g$ is ordinary at $p$ if and only if $\rhobarloc{g}$ is reducible.

\vspace{.1cm}

\item
Hypothesis (\ref{hyp:repn}) is equivalent to assuming that $\rhobarinert{f}$ is isomorphic to either $I(1)$ or $\psmallmat{\omega}{*}{0}{1} $ with $*$ neither 0 nor tr\`es-ramifi\'ee. 

\vspace{.1cm}

\item
Theorem \ref{thm:medweight} can fail for weights as low as $p^2+1$.  For example, there is a newform $f \in S_{10}(\Gamma_0(17))$ which is non-ordinary at $p=3$ and congruent to the unique normalized newform $g \in S_2(\Gamma_0(17))$. The form $g$ is non-ordinary at $3$, and thus $\rhobarloc{f} \cong \rhobarloc{g}$ is irreducible by Theorem \ref{thm:reps}.  In particular, hypotheses (\ref{hyp:irred}) and (\ref{hyp:repn}) are satisfied;  however, for this form, one computes that $\mm{f}{+} =1$. (We note that determining $\mm{f}{\pm}$ for a particular form $f$ is a finite computation.)

Possibly such counter-examples are common for following reason: let $g \in S_2(\Gamma,\Qpbar)$ denote any eigenform which satisfies hypotheses (\ref{hyp:irred}) and (\ref{hyp:repn}).  Consider $\theta^{p-1}(\vpb_{g})$
which is an eigensymbol in $H^1_c(\Gamma,\Vbar_{p^2-1})$.  By \cite[Proposition 2.5]{AS}, there exists an eigenform $f \in S_{p^2+1}(\Gamma,\Qpbar)$ whose system of Hecke-eigenvalues reduces to those of $\theta^{p-1}(\vpb_{g})$.  Thus, by Fermat's little theorem,
$$
\red{a_\ell(f)} = \ell^{p-1} \red{a_\ell(g)} = \red{a_\ell(g)}
$$ 
for all $\ell \neq p$.  Moreover, $\red{a_p(f)} = \red{a_p(g)}$ since both are 0.  Thus, $\vpb_f$ and $\theta^{p-1}(\vpb_{g})$ have the same system of Hecke-eigenvalues for the full Hecke-algebra.  A strong enough mod $p$ multiplicity one theorem (which is not currently known, and may not be always be true) would then imply equality of these two symbols up to a constant.   Thus, $\vpb_f$ is in the image of $\theta$, and  by Lemma \ref{lemma:mum}, we would then have that $\mm{f}{\pm} > 0$.
\vspace{.2cm}

\item The condition that $\rhobarloc{f}$ is not decomposable is  necessary.  For example, there is a newform $f \in S_{10}(\Gamma_0(21))$ which is non-ordinary at $5$ and congruent  to the unique normalized newform $g \in S_2(\Gamma_0(21))$.  In this example, $\rhobar_f$ is irreducible, $\rhobarloc{f}$ is decomposable, and $\mm{f}{\pm}>0$.

Possibly such counter-examples are again common for a similar reason as in the previous remark.  Take $g \in S_2(\Gamma,\Qpbar)$ with $\rhobar_g$ irreducible and $\rhobarloc{g}$ decomposable.  By Gross' tameness criterion \cite{Gross}, there exists a form $h \in S_{p-1}(\Gamma,\Qpbar)$ such that $\rhobar_h \otimes \omega \cong \rhobar_g$.  The associated eigensymbol $\vpb_h$ is in $H^1_c(\Gamma,V_{p-3}(\Zpbar))$, and thus $\theta(\vpb_h)$ is in $H^1_c(\Gamma,V_{2p-2}(\Zpbar))$.  By \cite[Proposition 2.5]{AS}, there exists $f \in S_{2p}(\Gamma,\Qpbar)$ whose system of Hecke-eigenvalues reduces to those of $\theta(\vpb_h)$.   In particular, $\rhobar_f \cong \rhobar_h \otimes \omega \cong \rhobar_g$, and thus $f$ satisfies hypotheses (\ref{hyp:irred}), $k(\rhobar_f)=2$, and $\rhobarloc{f}$ decomposable.

Note that $\theta(\vpb_h)$ and $\vpb_f$ have the same system of Hecke-eigenvalues.  Thus, as before, a strong enough mod $p$ multiplicity one result would give equality of these symbols up to a constant.  In particular, we would obtain that $\vpb_f$ is in the image of $\theta$, and   by Lemma \ref{lemma:mum},  $\mm{f}{\pm} > 0$.

\vspace{.1cm}

\item 
The question of determining the structure of $\rhobarloc{f}$ remains a difficult one.  Partial results exist when the weight $k$ is not too large.  For instance, if $k=p+1$ and $f$ is non-ordinary, then by a result of Edixhoven \cite{Edixhoven}, $\rhobarloc{f}$ is automatically irreducible and isomorphic to $I(1)$.  More recently, Berger \cite{Berger} showed that  if $k=2p$, then $\rhobarloc{f}$ is irreducible if and only if $\ord_p(a_p) \neq 1$.  Moreover,
$$
\rhobarinert{f} = 
\begin{cases}
I(1) & \text{~if~}0 < \ord_p(a_p) < 1 \\
I(2p-1) & \text{~if~}\ord_p(a_p) > 1 \\
\psmallmat{\omega}{*}{0}{1} 
\text{~or~}
\psmallmat{1}{*}{0}{\omega}   & \text{~if~}\ord_p(a_p) = 1 
\end{cases}
$$
Unfortunately, even in this small weight, we do not know how to determine which representation occurs in the last of these three cases solely from the value of $\ord_p(a_p)$, and, in particular, we cannot determine the value of $k(\rhobar_f)$.
\end{enumerate}
\end{remark}

In the following corollary we maintain the hypotheses and notation of Theorem \ref{thm:medweight}.

\begin{cor} 
\label{cor:invmedweight}
If $\rhobarloc{f}$ is reducible (resp.\ irreducible), then
\begin{enumerate}
\item 
\label{item:mumw}
$\mu(\tnfi) = 0$ for $n \gg 0 \iff \mu(g,\omega^i) = 0$ (resp.\ $\mu^\pm(g,\omega^i)=0$);
\vspace{.2cm}
\item if the equivalent conditions of (\ref{item:mumw}) hold and $n \gg 0$, then 
$$
\lam(\tnfi) =   
p^n - p^{n-1} +
\begin{cases}
 \lam(g,\omega^i) & \text{~if~}\rhobarloc{f} \text{~is~reducible}, \\
q_{n-1} + \lam^{\text{-}\ve_n}(g,\omega^i) & \text{~if~}\rhobarloc{f} \text{~is~irreducible}. 
\end{cases}
$$ 
\end{enumerate}
\end{cor}

\begin{proof}
We first note that $g$ is ordinary if and only if $\rhobarloc{f}$ is reducible (see Remark \ref{rmk:medweight}.\ref{part:gord}).
The corollary then follows from Theorem \ref{thm:medweight}, Theorem \ref{thm:wt2}, and Lemma \ref{lemma:nuninv}.
\end{proof}
\subsection{A key lemma}

The main tool in proving Theorem \ref{thm:medweight} is the map $\alpha$ of section \ref{sec:alpha}.  If $\alpha(\vpb_f) \neq 0$, then one can produce a congruence to a weight 2 form, and begin to compare their Mazur--Tate elements.  In this section, we establish the non-vanishing of $\alpha(\vpb_f)$ for the forms $f$ we are considering.

\begin{lemma}
\label{lemma:inj}
If $f$ satisfies the hypotheses of Theorem \ref{thm:medweight}, then
$\alpha(\vpb_f^\pm) \neq 0$.
\end{lemma}

\begin{proof}
First note that hypotheses (2) and (3) of Theorem~\ref{thm:medweight} imply that $k \equiv 2 \pmod{p-1}$.  The case $k=2$ is vacuous, and the case $k=p+1$ follows from \cite[Theorem 3.4(a)]{AS}.
For $k \geq 2p$, by \cite[Theorem 3.4(c)]{AS},
it suffices to show that $\vpb_f^\pm$ cannot lie in the image of the theta operator
$$\theta : H^1_c(\Gamma,\Vbar_{k-p-3}(1)) \to H^1_c(\Gamma,\Vbar_{k-2}).$$
We will prove this by showing that no eigensymbol in the image of $\theta$ has residual representation isomorphic to $\rhobar_f$ after restriction to
$I_p$.  

Assume first that $\rhobar_f|_{I_p}$ is irreducible and thus isomorphic to
$I(1)$.  For any weight $m \geq 2$, let
$\L^{\irr}(m)$ denote the set of $t \in \Z/(p^2-1)\Z$ such that
there exists an eigenform $g$ on $\Gamma$ of weight $m$ with
$\rhobar_g|_{I_p} \cong I(t)$.  ($\L^{\irr}(m)$ should really be regarded as
a subset of the quotient of $\Z/(p^2-1)\Z$ by the relation that
$t \sim pt$ for all $t$.)  Let $\L^{\irr}_{\theta}(m)$ denote the subset of 
$\L^{\irr}(m)$ of 
$t$ which occur for forms $g$ in the image of $\theta$.  We aim to show that $1,p \notin \L^{\irr}_{\theta}(m)$ for $m<p^2+1$.

By \cite[Theorem 3.4]{AS} and  Lemma \ref{lemma:nebenreps}, we have
\begin{align*}
\L^{\irr}_\theta(m) &= \emptyset \text{~~~~for $m \leq p+1$}; \\
\L^{\irr}_\theta(m) &\subseteq \bigl\{t + p + 1 \,;\, t \in \L^{\irr}(m-p-1)
\bigr\} \text{~~~~for $m > p+1$}; \\
\L^{\irr}(m) &\subseteq \L^{\irr}_\theta(m) \cup \{ m' + 1\}
\end{align*}
where $m'$ denotes the remainder when $m-2$ is divided by $p-1$.
It now follows by a straightforward induction that for $k < p^2+1$,
$k \equiv 2 \pmod{p-1}$, one has
$$\L^{\irr}_\theta(k) = \left\{ j(p-1) + 1 \,;\, 2 \leq j \leq \frac{k-2}{p-1},
j \neq \frac{p+3}{2} \right\}.$$
(Note that the induction involves all even $k$, not just those which are
congruent to $2$ modulo $p-1$.)
In particular, neither $1$ nor $p$ lies in $\L^{\irr}_\theta(k)$ for such
$k$; it follows that $\vpb_f^\pm$ does not lie in the image of $\theta$, as desired.

When $\rhobar_f|_{I_{p}}$ is reducible it is necessarily isomorphic to
$\psmallmat{\omega}{*}{0}{1} $ with $*$ non-zero.  Let
$\L^{\rd}(m) \subseteq \Z/(p-1)\Z$ denote the set of all $t$ such that
$\psmallmat{\omega^t}{*}{0}{*}$ (with $*$ non-zero)
can occur as the restriction to $I_p$ of the residual representation of
some form of weight $k$ for $\Gamma$ which has a globally irreducible
residual representation.  As before, we have
\begin{align*}
\L^{\rd}_\theta(m) &= \emptyset \text{~~~~for $m \leq p+1$}; \\
\L^{\rd}_\theta(m) &\subseteq \bigl\{t + 1 \,;\, t \in \L^{\rd}_\theta(m-p-1)
\bigr\} \text{~~~~for $m > p+1$}; \\
\L^{\rd}(m) &\subseteq \L^{\rd}_\theta(m) \cup \{ m-1 \}.
\end{align*}
Once again, a straightforward induction establishes that
for $k \leq \frac{p^2+3}{2}$, $k \equiv 2 \pmod{p-1}$ we have
$$\L^{\rd}_\theta(k) \subseteq \left\{
-j \,;\, 0 \leq j \leq \frac{k-2}{p-1}-2 \right\}$$
while for $\frac{p^2+3}{2} < k < p^2+1$,
$k \equiv 2 \pmod{p-1}$ we have
$$\L^{\rd}_\theta(k) \subseteq \left\{
 -j \,;\, 0 \leq j \leq \frac{k-2}{p-1}-3 \right\}.$$
In particular, $1$ does not lie in $\L^{\rd}_\theta(k)$ for such $k$,
so that $\vpb_f^\pm$ does lie in the image of $\theta$.
\end{proof}

\subsection{Proof of Theorem \ref{thm:medweight}}
Let $\varphi_f \in H^1_c(\Gamma,V_{k-2}(\O))$ denote the modular symbol attached to $f$, and let $\vpb_f$ denote its non-zero image in $H^1_c(\Gamma,\Vbar_{k-2})$.  By Lemma \ref{lemma:inj},  $\alpha(\vpb^\pm_f)$ is non-zero. Thus, by Lemma \ref{lemma:mum}, $\mm{f}{\pm} = 0$ which establishes the first part of the theorem.

By Lemma \ref{lemma:inj}, $\alpha(\vpb_f)$ is a (non-zero) eigensymbol in $H^1_c(\Gamma_0,\F)$ with the same Hecke-eigenvalues as $\vpb_f$ for all primes $\ell$ (even $\ell=p$).  By \cite[Proposition 2.5]{AS}, there exists an eigenform $h \in S_2(\Gamma_0)$ whose Hecke-eigenvalues reduce to the eigenvalues of $\vpb_f$.   Since $\vpb_f$ is \nop, the same is true of $h$.  However, this implies that $h$ must be old at $p$; indeed, any form of weight 2 which is $p$-new is automatically $p$-ordinary.  Let $g \in S_2(\Gamma)$ denote the corresponding eigenform which is new at $p$, but has all same Hecke-eigenvalues at primes away from $p$; that is, $h$ is in the span of $g(z)$ and $g(pz)$. 

Let $\vpb_g$ in $H^1_c(\Gamma,\F)$ denote the reduction of the modular symbol attached to $g$.  One might except a congruence between $\vpb_g$ and $\alpha(\vpb_f)$.  However, the former symbol has level $\Gamma$ while the latter has level $\Gamma_0$.  If we view $\vpb_g$ in $H^1_c(\Gamma_0,\F)$, then it is no longer an eigensymbol at $p$.  Instead, we consider the symbol
 $\vpb_g \big| \psmallmat{p}{0}{0}{1}$  in $H^1_c(\Gamma_0,\F)$ which is also an eigensymbol at all primes away from $p$, and moreover,
\begin{align*}
\left( \vpb_g \big| \psmallmat{p}{0}{0}{1} \right) \big| U_p
&= \sum_{a=0}^{p-1} \left( \vpb_g \big| \psmallmat{p}{0}{0}{1} \right) \big| \psmallmat{1}{a}{0}{p}  \\
&= \sum_{a=0}^{p-1}  \vpb_g \big| \psmallmat{p}{pa}{0}{p}  
= \sum_{a=0}^{p-1}  \vpb_g \big| \psmallmat{1}{a}{0}{1}  \\
&= \sum_{a=0}^{p-1}  \vpb_g  = p \cdot \vpb_g = 0.
\end{align*}
Thus, $\vpb_g \big| \psmallmat{p}{0}{0}{1}$ is a Hecke-eigensymbol for the full Hecke-algebra.  As the same is true of $\alpha(\vpb_f)$, by mod $p$ multiplicity one (see \cite[Theorem 2]{Ribet-Brazil}), we have 
$$
\alpha(\vpb_f^\pm) = c^\pm \cdot \vpb_g^\pm \Dp
$$
with $c^\pm \neq 0$.  Moreover, by changing $\Omega^\pm_f$ by a $p$-unit, we can take $c^\pm$ equal to 1.  Then, by Lemmas \ref{lemma:alphastick} and \ref{lemma:degen},
$$
\red{\ttnf} = {\ttn{\alpha(\vpb_f)}} = \ttn{\vpb_g \Dp} =\nun{n}{n-1} \bigl(\red{\tt{n-1}{\varphi_g}}\bigr)
$$
completing the proof of theorem.

\section{Results in small slope}
\label{sec:smallslope}

In this section, we will prove a theorem along the lines of Theorem \ref{thm:medweight}, but instead of assuming a bound on the weight of $f$, we assume on bound on its slope.  Interestingly, the proof uses a congruence argument even though the $\mu$-invariants that appear need not be zero.

\subsection{Statement of theorem}

\begin{thm}
\label{thm:lowslope}
Let $f$ be an eigenform in $S_k(\Gamma,\Qpbar)$ such that 
\begin{enumerate}
\item $\rhobar_f$ is irreducible,
\vspace{.1cm}
\item 
\label{hyp:slope}
$0 < \ord_p(a_p) < p-1$,
\vspace{.1cm}
\item $k(\rhobar_f) = 2$ and $\rhobarloc{f}$ is not decomposable.
\end{enumerate}
Then 
\begin{enumerate}
\item $\mm{f}{\pm} \leq \ord_p(a_p)$ holds for both choices of sign;
\vspace{.2cm}
\item there exists an eigenform $g \in S_2(\Gamma)$ with $\red{a_\ell(f)} = \red{a_\ell(g)}$ for all primes $\ell \neq p$, and a choice of cohomological  periods $\Omega_f, \Omega_g \in \C$ such that 
$$
\red{\varpi^{-a}\ttv{n}{f}{i}} =  \nun{n}{n-1}\bigl(\red{ \ttv{n-1}{g}{i}}\bigr) ~\text{~in~~} \F[G_n]
$$ 
where $a \in \Z^{\geq0}$ is such that $\ord_p(\varpi^a)=\mm{f}{\ve_i}$.
\end{enumerate}
\end{thm}

We maintain the hypotheses and notation of Theorem \ref{thm:lowslope} in the following corollary.

\begin{cor} 
\label{cor:invlowslope}
If $\rhobarloc{f}$ is reducible (resp.\ irreducible), then
\begin{enumerate}
\item 
\label{item:muls}
$\mu(\tnfi) = \mm{f}{\ve_i}$ for $n \gg 0\iff\mu(g,\omega^i) = 0$ (resp.\ $\mu^\pm(g,\omega^i)=0$).
\vspace{.2cm}
\item 
if the equivalent conditions of (\ref{item:muls}) hold and $n \gg 0$, then 
$$
\lam(\tnf) =   
p^n - p^{n-1} +
\begin{cases}
\lam(g,\omega^i) & \text{~if~}\rhobarloc{f} \text{~is~reducible}, \\
q_{n-1} + \lam^{\text{-}\ve_n}(g,\omega^i) & \text{~if~}\rhobarloc{f} \text{~is~irreducible}.
\end{cases}
$$ 
\end{enumerate}
\end{cor}

\begin{remark}
\noindent
\begin{enumerate}
\item By results of Buzzard and Gee \cite{BG},  if $k \equiv 2 \pmod{p-1}$ and $\ord_p(a_p) <1$, then $\rhobarinert{f} \cong I(1)$ and thus hypotheses (\ref{hyp:irred}) and (\ref{hyp:repn}) are automatic.

\item
Hypothesis (\ref{hyp:slope}) is necessary as we have found forms of slope $p-1$ whose $\lambda$-invariants do not follow the pattern described by Corollary \ref{cor:invlowslope}. In these examples, the $\lambda$-invariants satisfy
$$
\lambda(\tnf) = p^n - p^{n-2} + \begin{cases} 
\lambda(g) & \text{~if~}\rhobarloc{f}\text{~is~reducible,} \\
q_{n-2} + \lambda^{\ve_n}(g) & \text{~if~}\rhobarloc{f}\text{~is~irreducible,} 
\end{cases}
$$
for $n \gg 0$ where $g$ is some congruent form in weight 2.  This phenomenon will be further explored in section \ref{sec:example}.
\end{enumerate}
\end{remark}

\subsection{Filtration lemmas}
\label{sec:filt}

Let $\O$ be the ring of integers in a finite extension of $\Qp$, and consider the filtration on $V_g(\O)$ given by
$$
\FV{r} = \Fil^r(V_g(\O)) = 
\left\{ \sum_{j=0}^g b_j X^j Y^{g-j} \in V_g(\O) ~:~
p^{r-j} \mid b_j \text{~for~} 0 \leq j \leq r-1
\right\}.
$$

Recall that the semi-group $S_0(p)$ appearing in the next lemma was introduced in section \ref{sec:alpha}.

\begin{lemma} We have:
\label{lemma:filt}
\begin{enumerate}
\item $\FV{r}$ is stable under the action of $S_0(p)$.
\item If $P \in \FV{r}$, then $P \big| \psmallmat{1}{a}{0}{p} \in p^r V_g(\O)$.
\end{enumerate}
\end{lemma}

\begin{proof}
For $\psmallmat{a}{b}{c}{d} \in S_0(p)$, we have
$$
X^j Y^{g-j} \big| \psmallmat{a}{b}{c}{d} =
(dX-cY)^j (-bX+aY)^{g-j}.
$$
Expanding the above expression and using the fact that  $p \mid c$ and $p \nmid a$, one sees that the coefficient of $X^{s}Y^{g-s}$ is divisible by $p^{j-s}$ for $s \leq j$.  The first part of the lemma follows from this observation.

For the second part, we have that 
$$
p^{r-j} X^j Y^{g-j} \big| \psmallmat{1}{a}{0}{p} =
p^{r-j} (pX)^j (-aX+Y)^{g-j} \in p^r V_g(\O)
$$
which proves the lemma.
\end{proof}

\begin{lemma}
\label{lemma:filtslope}
If $\varphi \in H^1_c(\Gamma,V_g(\O))$ is a $T_p$-eigensymbol which takes values in $\FV{r}$ and $|| \varphi || =1$, then the slope of $\varphi$ is greater than or equal to $r$.
\end{lemma}

\begin{proof} 
Write $\varphi \big| T_p = \lambda \cdot \varphi$, and choose $D \in \Delta_0$ such that $||\varphi(D)||=1$.  We then have
\begin{equation}
\label{eqn:up}
\lambda \cdot \varphi(D) = (\varphi \big| T_p)(D) = \sum_{a=0}^{p-1} \varphi\left(\psmallmat{1}{a}{0}{p} D\right)\big| \psmallmat{1}{a}{0}{p} + \varphi\left(\psmallmat{p}{0}{0}{1}D\right) \Dp.
\end{equation}
By Lemma \ref{lemma:filt}, all of the terms on the right-hand side 
are divisible by $p^r$ except for possibly the last.

To deal with the final term write $\varphi\left(\psmallmat{p}{0}{0}{1}D\right) = \sum_{j=0}^g a_j X^j Y^{g-j}$.  By Lemma \ref{lemma:FE}, we have that
$$
 \sum_{j=0}^g a_j X^j Y^{g-j} \Big| \psmallmat{0}{-1}{N}{0} = 
 \sum_{j=0}^g a_j (-NY)^j X^{g-j} =
 \sum_{j=0}^g (-1)^j N a_{g-j} X^j Y^{g-j} 
$$
is also a value of $\varphi$, and thus is in $\FV{r}$.  In particular, for $0 \leq j \leq r$, we have $p^{r-j} \mid Na_{g-j}$, and hence $p^{r-j} \mid a_{g-j}$ as $\gcd(N,p)=1$.  Finally,
$$
\varphi \left(\psmallmat{p}{0}{0}{1}D\right) \Big| \psmallmat{p}{0}{0}{1} = \sum_{j=0}^g a_j X^j Y^{g-j} \Big| \psmallmat{p}{0}{0}{1} = 
\sum_{j=0}^g a_j p^{g-j} X^j {Y}^{g-j}
$$
is in $p^r V_g$, and thus $\lambda \cdot \varphi(D) \in p^r V_g$.  As $||\varphi(D)||=1$, we deduce that $\ord_p(\lambda) \geq r$ as desired.
\end{proof}

In what follows, we will need to make use of a finer filtration on $V_g(\O)$.  Note that as an abelian group, $\FV{r}/\FV{r+1}$
is simply $(\O/p\O)^{r+1}$.  Thus, we introduce the following subfiltration of $\FV{r}$; for $s \leq r$ we set
\begin{align*}
\FV{r,s} =   
\left\{ \sum_{j=0}^g b_j X^j Y^{g-j} \in \FV{r} ~:~
p^{r-j+1} \mid b_j \text{~for~} r+1-s \leq j \leq r 
\right\}.
\end{align*}
Note that
$$
\FV{r} = \FV{r,0} \supsetneq \FV{r,1} \supsetneq \dots
\supsetneq \FV{r,r} \supsetneq \FV{r,r+1} = \FV{r+1}.
$$

In the following lemma, $\left( \O/p\O(a^j) \right)(r)$ denotes the $S_0(p)$-module $\O/p\O$ on which $\gamma = \psmallmat{a}{b}{c}{d}$ acts by multiplication by $\det(\gamma)^r \cdot a^j$.  

\begin{lemma}
\label{lemma:quot}
\noindent
\begin{enumerate}
\item 
$\FV{r,s}$ is stable under the action of $S_0(p)$.
\item For $0 \leq s \leq r$,
$$
\FV{r,s} / \FV{r,s+1} \cong \left( \O/p\O(a^{g-2r+2s}) \right) (r-s)
$$
as $S_0(p)$-modules.  Moreover, this quotient is generated by the image of the monomial $p^s X^{r-s} Y^{g-r+s}$.
\end{enumerate}
\end{lemma}

\begin{proof}
The first part follows just as in Lemma \ref{lemma:filt}.  For the second part, directly from the definitions, we have that $\FV{r,s} / \FV{r,s+1}$ is isomorphic to $\O/p\O$ and is generated by the image of $p^s X^{r-s} Y^{g-r+s}$.  For the $S_0(p)$-action, we have
\begin{align*}
p^s X^{r-s} Y^{g-r+s} ~\big| \psmallmat{a}{b}{c}{d} 
&\equiv p^s (dX)^{r-s} (-bX+aY)^{g-r+s} &&\pmod{\FV{r,s+1}} \\ 
&\equiv d^{r-s} a^{g-r+s} \cdot p^s X^{r-s} Y^{g-r+s} &&\pmod{\FV{r,s+1}}. 
\end{align*}
Thus,
\begin{align*}
\FV{r,s} / \FV{r,s+1} &\cong \O/p\O(d^{r-s} a^{g-r+s})\\
&\cong \O/p\O( (ad)^{r-s} a^{g-2r+2s} ) 
\cong \left(\O/p\O\left(a^{g-2r+2s}\right)\right)(r-s)
\end{align*}
as desired.
\end{proof}

The following is a slight refinement of Lemma \ref{lemma:filtslope}, and will be useful in the proof of Theorem \ref{thm:lowslope}.

\begin{lemma}
\label{lemma:filtslopestrong}
Let $\varphi \in H^1_c(\Gamma,V_g(\O))$ be a $T_p$-eigensymbol which takes values in $\FV{r,r}$ and such that $r \leq \mum(\varphi) < r+1$.  If $|| \varphi || =1$, then the slope of $\varphi$ is greater than or equal to $\mum(\varphi)$.
\end{lemma}

\begin{proof}
The proof follows just as in Lemma \ref{lemma:filtslope}.
\end{proof}

\subsection{Proof of Theorem \ref{thm:lowslope}}

To ease notation, set $F^a = \Fil^a(V_{k-2}(\O))$, $F^{a,b}=\Fil^{a,b}(V_{k-2}(\O))$, and $\varphi = \varphi_f^{\pm}$.  Let $r \geq 0$ denote the largest integer such that $\varphi$ takes values in $F^{r}$, and let $s \geq 0$ denote the largest integer such that $\varphi$ takes values in $F^{r,s}$. Note that by definition $s \leq  r$ since $F^{r,r+1} = F^{r+1}$.

By Lemma \ref{lemma:filtslope}, we have $r \leq \ord_p(a_p)$, and thus hypothesis (\ref{hyp:slope}) gives
\begin{equation}
\label{eqn:ineq}
s \leq r < p-1.
\end{equation}
Our first goal is to show that $r=s$.

Since $\varphi$ does not take all of its values in $F^{r,s+1}$, its image in $H^1_c(\Gamma_0,F^{r,s}/F^{r,s+1})$ is non-zero. Thus, by Lemma \ref{lemma:quot}, $\varphi$ gives rise to a non-zero eigensymbol in 
$$
H^1_c(\Gamma_0,\O/p\O(a^{p-1-2r+2s}))(r-s);
$$
here, we are using that $k \equiv 2 \pmod{p-1}$.
Finally, if this symbol takes values in $\varpi^t \O$ but not in $\varpi^{t+1}\O$,  projecting modulo $\varpi^{t+1}$ and dividing by $\varpi^t$ gives rise to a non-zero eigensymbol 
$$
\o{\eta}_f \in H^1_c(\Gamma_0,\F(a^{p-1-2r+2s}))(r-s).
$$
Then, by \cite[Proposition 2.5 and Lemma 2.6]{AS}, there exists an eigenform $g$ in $S_2(\Gamma_1, \omega^{-2r+2s})$ such that $\rhobar_f \cong \rhobar_g \otimes \omega^{r-s}$.

By Lemma \ref{lemma:nebenreps}, we then have one of the following three possibilities:
\begin{equation}
\rhobarinert{g} = 
I((-2r+2s)'+1),~
\pmat{\omega^{-2r+2s+1}}{*}{0}{1},~\text{~or~}
\pmat{\omega}{*}{0}{\omega^{-2r+2s}}
\end{equation}
where, for an integer $x$, we set $x'$ equal to the unique integer $j$ with $0 \leq j \leq p-2$ and $j \equiv x \pmod{p-1}$.  In the locally reducible case, we then have
\begin{equation}
\label{eqn:reps}
\rhobarinert{f} = 
\pmat{\omega^{-r+s+1}}{*}{0}{\omega^{r-s}} \text{~or~} \pmat{\omega^{r-s+1}}{*}{0}{\omega^{-r+s}}.
\end{equation}
Hypothesis (\ref{hyp:repn}) implies that $\rhobarinert{f} \cong \psmallmat{\omega}{*}{0}{1}$ with $*$ non-zero, and thus $s \equiv r \pmod{p-1}$.  The bound in (\ref{eqn:ineq}) then gives $s=r$ as desired.

If $\rhobarloc{f}$ is irreducible, we have
$$
\rhobarloc{f} = 
\begin{cases}
I(p+(p-1)(r-s)) &  \text{if~}r-s\leq\frac{p-1}{2},  \\
I(2p-1+(p-1)(r-s)) &  \text{if~}r-s>\frac{p-1}{2}.  
\end{cases}
$$
By hypothesis (\ref{hyp:repn}), we then have
$$
\rhobarloc{f} \cong I(1) \cong I(p+(p-1)(r-s)) ~\text{~or~}~ I(2p-1+(p-1)(r-s)).
$$  
In the first case,
$$
\label{eqn:cong}
p+(p-1)(r-s) \equiv 1 \text{~or~} p \pmod{p^2-1}.
$$
Thus,
$$
r-s \equiv p \text{~or~} 0 \pmod{p+1}
$$
which forces $s=r$. A similar analysis in the second case shows that no such $r$ and $s$ exist.
Hence, in all possible cases, $r=s$ and $\varphi_f$ takes values in $F^{r,r}$.  

By Lemma \ref{lemma:quot}, $p^rY^{k-2}$ generates $F^{r,r} / F^{r,r+1}$, and thus the image of $\varphi$ in 
$$
H^1_c(\Gamma_0,F^{r,r}/F^{r,r+1}) \cong H^1_c(\Gamma_0,\O/p\O)
$$ 
is given by
$$
D \mapsto \left(\frac{1}{p^r}\varphi(D)\big|_{(X,Y)=(0,1)}\right) \pmod{p}.
$$
This implies that $\o{\eta}_f$ is given by 
$$
D \mapsto \left(\frac{1}{\varpi^t p^r}\varphi(D)\big|_{(X,Y)=(0,1)}\right) \pmod{\varpi}.
$$
By construction, $\ord_p(\varpi^t p^r)=\mm{f}{\pm}$; thus, if we let $a$ be the integer such that $\ord_p(\varpi^a)=\mm{f}{\pm}$, scaling by a unit then yields the eigensymbol
$$
D \mapsto \o{\frac{1}{\varpi^a}\varphi(D)\big|_{(X,Y)=(0,1)}} \in \F
$$
in $H^1_c(\Gamma_0,\F)$ whose system of Hecke-eigenvalues is the reduction of the system of eigenvalues attached to $f$.

The argument now proceeds as in Theorem \ref{thm:medweight} to show that
$$
\red{\varpi^{-a}\tnfi} =  \nun{n}{n-1}(\red{ \ti{n-1}{g}}) ~~\text{~in~} \F[G_n]
$$ 
as desired.  

Lastly, the inequality $\mm{f}{\pm} \leq \ord_p(a_p)$ follows from Lemma \ref{lemma:filtslopestrong}.

\section{A strange example}
\label{sec:example}

In this section, we describe a strange behavior of Iwasawa invariants of forms which do not satisfy the hypotheses of Theorems \ref{thm:medweight} and \ref{thm:lowslope}.  

Take $p=3$, and consider the space of cuspforms $S_{18}(\Gamma_0(11),\Qpbar)$.  In this space, there are exactly 2 (Galois conjugacy classes of) eigenforms of slope 2 whose residual representations are isomorphic to the 3-torsion on $X_0(11)$.  Let $f_1$ and $f_2$ denote representatives from the each of these classes.  Note that the associated residual representation is locally reducible at $3$ as $X_0(11)$ is ordinary at 3.  Neither Theorem \ref{thm:medweight} nor Theorem \ref{thm:lowslope} apply directly to these forms as the weight $k=18$ is greater than $p^2$, and the slope $2$ is not less than $p-1$.  

Let $\O_j$ denote the ring of integers of the field generated by the coefficients of $f_j$, and let $\varpi_j$ denote a uniformizer.
A computer computation shows that $\mump(f_j)=2$ for $j=1,2$, and so we consider the map
\begin{align*}
V_{16}(\O_j) &\lra \O_j/p^2\varpi_j\O_j\\
P(X,Y) &\mapsto P(0,1) \pmod{p^2\varpi_j}.
\end{align*}
As this map is $\Gamma_0(p^3)$-equivariant, it induces a Hecke-equivariant map
$$
\alpha: H^1_c(\Gamma,V_{16}(\O_j)) \lra H^1_c(\Gamma_0(p^3N),\O_j/p^2\varpi_j\O_j).
$$
By construction, $\alpha(\varphi_{f_j}^+)$ is non-zero and takes values in $p^2\O_j/p^2\varpi_j\O \cong \F_3$.  If we view $\alpha(\varphi_{f_j}^+)$ in $H^1_c(\Gamma_0(p^3N),\F_3)$, then it is an eigensymbol whose system of Hecke-eigenvalues is the reduction of the system attached to $f_j$.  

A computer computation then shows that the subspace of $H^1_c(\Gamma_0(p^3N),\F_3)^+$ with this system of Hecke-eigenvalues is 3-dimensional and generated by 
$$
\vpb_{g}^+ \Dpv{p},~ \vpb_{g}^+ \Dpv{p^2}, \text{~~and~~} \vpb_{g}^+ \Dpv{p^3}
$$  
where $g$ is the unique normalized eigenform in $S_2(\Gamma_0(11))$.  (Note that mod $p$ multiplicity one is failing for trivial reasons!)
Thus, we have
$$
\alpha(\varphi_{f_j}^+) = a_{j,1} \cdot \vpb_{g}^+ \Dpv{p} + a_{j,2} \cdot \vpb_{g}^+ \Dpv{p^2} + a_{j,3} \cdot \vpb_{g}^+ \Dpv{p^3},
$$
and, in particular,
\begin{equation}
\label{eqn:red2}
\red{\frac{\tn{f_j}}{p^2} } = a_{j,1} \cdot \nun{n}{n-1}(\t{n-1}{g}) + 
a_{j,2} \cdot \nun{n}{n-2}(\t{n-2}{g}) +
a_{j,3} \cdot \nun{n}{n-3}(\t{n-3}{g}).
\end{equation}
These equations should allow us to determine the Iwasawa invariants of $f_j$ in terms of the invariants of the $p$-ordinary form $g$; in this case, one computes that $\mu(g)=\lambda(g)=0$. 

A key difference now emerges between $f_1$ and $f_2$; namely, a computer computation shows that 
$$
a_{1,1} \neq 0 \text{~~while~~} a_{2,1} = 0.
$$
This vanishing is significant because for $j=1$, the first term on the right hand side of (\ref{eqn:red2}) dominates in calculating $\lambda$, and we have
$$
\lambda(\tn{f_1}) = p^n-p^{n-1} + \lambda(g) = p^n-p^{n-1}.
$$
For $j=2$, the second term in (\ref{eqn:red2}) dominates and we have
$$
\lambda(\tn{f_2}) = p^n-p^{n-2} + \lambda(g) = p^n-p^{n-2}.
$$
Thus, there is a ``second-order" difference in the rate of growth of the $\lambda$-invariants of $f_1$ and $f_2$.

Similar examples exist in the locally irreducible case.  For instance, for $p=3$, there is an eigenform $f$ in $S_{18}(\Gamma_0(17),\Qpbar)$ whose residual representation is isomorphic to the 3-torsion in $X_0(17)$ (which is locally irreducible at 3 as $X_0(17)$ is supersingular at 3), whose slope is 5, and for which we have
$$
\lambda(\tnf) = p^n-p^{n-2} + q_{n-2} 
$$
as opposed to the $\lambda$-invariants $q_{n+1} = p^n-p^{n-1} + q_{n-1}$ which occur in Theorems \ref{thm:medweight} and \ref{thm:lowslope}.

\end{document}